\documentclass[journal]{IEEEtran}
\usepackage{amsmath}
\usepackage{amsfonts,amssymb}
\usepackage{CJK}
\usepackage{multicol}
\usepackage{multirow}
\usepackage{diagbox}
\usepackage{graphicx}
\usepackage{subfig} 
\usepackage{lineno} 
\allowdisplaybreaks % allow the align display in different pages
\usepackage{epstopdf} % transfer the figure with eps formation to figure with pdf formation
\usepackage{algorithm}
\usepackage{algorithmic}
\usepackage{pgf,tikz}
\usepackage{mathrsfs}
\usepackage{tikz}
\usepackage{enumerate}
\usepackage{makecell}
\usepackage{booktabs}       % 用于三线表规范
\usepackage{threeparttable} % 用于表格脚注和注释
\usepackage{graphicx}     
\usepackage[utf8]{inputenc}
\usepackage{subfig} 
\usepackage{bm}
\usepackage{subfloat}
\usepackage{colortbl,booktabs}%�ڶ����������˼���*rule
\usepackage{cite}
\usepackage{etoolbox}\AtBeginEnvironment{algorithmic}{\small}

\usepackage[subfigure]{tocloft}
\usetikzlibrary{arrows,shapes,snakes,shadows,positioning,automata,patterns}
\usetikzlibrary{trees,decorations.pathmorphing,decorations.markings}
\usetikzlibrary{backgrounds}
\usepackage[%dvipdfm,  %pdflatex,pdftex %�������������ļ��ķ�ʽ��ͬ
            pdfstartview=FitH,
            CJKbookmarks=true,
            bookmarksnumbered=true,
            bookmarksopen=true,
            colorlinks, %ע�͵������򽻲�����Ϊ��ɫ�߿�(��colorlinks��pdfborderͬʱע�͵�)
            pdfborder=001,   %ע�͵������򽻲�����Ϊ��ɫ�߿�
            linkcolor=blue,
            anchorcolor=blue,
            citecolor=blue
            ]{hyperref}
\newtheorem{theorem}{Theorem}
\newtheorem{assumption}{Assumption}

\newtheorem{lemma}{Lemma}
\newtheorem{corollary}{Corollary}

\newtheorem{remark}{Remark}

\usepackage{tcolorbox}
\makeatletter
\newcommand{\pushright}[1]{\ifmeasuring@#1\else\omit\hfill$\displaystyle#1$\fi\ignorespaces}
\newcommand{\pushleft}[1]{\ifmeasuring@#1\else\omit$\displaystyle#1$\hfill\fi\ignorespaces}
\makeatother
\usepackage{makecell, multirow, graphicx}

\newbool{submitA} \setbool{submitA}{false} % true 为版本 A，false 为版本 B

\definecolor{cuhkpl}{RGB}{152,24,147}

\def\BibTeX{{\rm B\kern-.05em{\sc i\kern-.025em b}\kern-.08em
		T\kern-.1667em\lower.7ex\hbox{E}\kern-.125emX}}
\begin{document}
	\title{Dynamic Regret and Tightened Constraint Violation for Communication-Efficient Distributed Online\\ Bilevel Optimization}
	\author{Fan Li,~Xinlei Yi,~Guanghui Wen,~Yang Shi,~Tianyou Chai and Tao Yang
			\thanks{This work was supported in part by the National Natural Science Foundation of China under Grant 62133003. (Corresponding author: Tao Yang.) }
\thanks{F. Li, T. Chai, and T. Yang are with the State Key Laboratory of Synthetical Automation for Process Industries, Northeastern University, Shenyang 110819, China (e-mail: 2410314@stu.neu.edu.cn; tychai@mail.neu.edu.cn; yangtao@mail.neu.edu.cn).}
		\thanks{X. Yi is with the Department of Control Science and Engineering, College of Electronics and Information Engineering, Tongji University, Shanghai 201804, China (e-mail: xinleiyi@tongji.edu.cn).}
		\thanks{G. Wen is with the Department of Systems Science, School of Mathematics, Southeast University, Nanjing 210096, China (e-mail: wenguanghui@gmail.com).}
		\thanks{Y. Shi is with the Department of Mechanical Engineering, University of
			Victoria, Victoria, BC V8W 2Y2, Canada (e-mail: yshi@uvic.ca).}
	}
\maketitle

\begin{abstract}
This paper investigates the distributed online bilevel optimization problem with time-varying inequality constraints. In this setting, a multi-agent network is confronted with a hierarchical decision-making task characterized by coupling between the inner and outer optimization levels. We first propose a communication-efficient single-loop distributed online primal--dual algorithm based on communication compression, which tracks both the inner-level solution and the Hessian-inverse-vector product using only one update per-round. We show that this algorithm achieves a dynamic  network regret bound of $\mathcal{O}(T^{\max\{\theta_1,1-\theta_1\}}+T^{\theta_1 } P_T+H_{1,T}+E_{1,T})$, where~$\theta_1 \in (0,1)$, $P_T$ is the outer-level path variation, $H_{1,T}$ is the inner-level path variation, and $E_{1,T}$ is the gradient path variation alongside a network cumulative constraint violation bound of  $\mathcal{O}(T^{1-\theta_1/2})$, which is further reduced to $\mathcal{O}(T^{1-\theta_1})$ under Slater's condition. Furthermore, to eliminate the dependence of regret on the inner-level and gradient path variation, we propose a novel adaptive multi-loop distributed online algorithm to control the inner-level approximation accuracy at each round. By dynamically adjusting the number of inner-level updates based on a gradient-norm tolerance condition,  this algorithm yields an improved dynamic network regret bound of 
$\mathcal{O}(T^{\max\{\theta_1,1-\theta_1\}}+T^{\theta_1 } P_T)$ which depends solely on the outer-level path variation while maintaining the same network cumulative constraint violation bound. Finally, numerical simulations verify the effectiveness of the proposed algorithms.
\end{abstract}


\IEEEpeerreviewmaketitle
\section{Introduction}
Bilevel optimization has recently garnered extensive attention in the fields of machine learning and control as a rigorous mathematical paradigm for characterizing hierarchical decision making. Many complex tasks in reality, such as meta-learning~\cite{KKmetelearning}, hyperparameter optimization~\cite{KKHyperOpti}, and reinforcement learning~\cite{KKRL} often exhibit a natural bilevel hierarchical structure.

Driven by the explosive growth of data and distributed computational resources, extending classical centralized bilevel optimization to distributed settings has become a compelling requirement. Specifically, the authors of~\cite{chen2023decentralized111} propose the MA-DSBO algorithm utilizing the Hessian-inverse-vector product (HvP) technique without explicit computation of full Hessian or Jacobian matrices. The authors of~\cite{DisBilAuto2023penalty} develop the decentralized alternating gradient method by leveraging the penalty function method. To eliminate inner computational loops, the authors of~\cite{niu2025lopa} present a single-loop personalized distributed algorithm that characterizes the impact of data heterogeneity. Furthermore, the authors of~\cite{niu2026hessian} combine time scale separation techniques to design a Hessian-free and single-loop algorithm for nonconvex-strongly convex settings. However, the aforementioned algorithms rely on static time-invariant cost functions, making them unsuitable for real-world applications in highly dynamic environments. For example, hierarchical resource allocation determines local allocations under an upper-level control parameter~\cite{ji2023lower}, while time-varying demands and resource availability motivate online costs and constraints~\cite{chen2017online}.

To address the case of dynamic environments, significant strides have been made in online bilevel optimization~\cite{OnlienBI11,OnlineBereNips33,OnlineFirAcx,Onlinelocalregret}. For example, pioneer work~\cite{OnlienBI11} formulates the online bilevel optimization framework to achieve sublinear regrets. Furthermore, the authors of~\cite{OnlineBereNips33} establish the online algorithm in the nonconvex-strongly convex setting with the window-averaging strategy, which aggregates information over a recent time window to mitigate temporal fluctuations. The authors of~\cite{OnlineFirAcx} utilize the Lagrangian reformulation to recast the online bilevel problem into single-level constrained optimization, thereby developing algorithms relying exclusively on first-order information. The authors of~\cite{Onlinelocalregret} develop novel efficient algorithms for online bilevel optimization that attain optimal regret bounds by dynamically adjusting inner iterations to track subproblem optimum drift and employing the window-averaging strategy.

In many networked applications, however, the data and computational resources involved in bilevel decision making are naturally distributed among multiple agents and vary over time.  These settings motivate the study of distributed online bilevel optimization.   
Existing bilevel optimization algorithms typically rely on Neumann series or conjugate gradient methods to approximate the Hessian inverse matrices with high precision~\cite{Hong2023two,Mengdi2018,CG11}. These approaches inevitably introduce computationally expensive nested loops for approximating the inner-level solution and the Hessian inverse matrices at each iteration. It is also not yet clear whether the approaches in~\cite{niu2025lopa,niu2026hessian} based on single-loop update mechanisms to obviate these computational loops can be generalized to online settings. Moreover, unlike single-level online optimization~\cite{KaiPseTac}, the dynamic network regret in online bilevel optimization typically depends on inner-level and gradient path variation~\cite{OnlienBI11,OnlineBereNips33,OnlineFirAcx,Onlinelocalregret}. These schemes fail to guarantee sublinear regret when severe environmental changes cause significant inner-level path variation. This motivates us to explore novel approaches to eliminate this dependence on inner-level path variation.

Practical physical systems are frequently subjected to complex time-varying inequality constraints, for example, dynamically varying capacity, power, and resource constraints in networked systems. 
To evaluate constraint satisfaction in online settings, cumulative constraint violation, which independently penalizes instantaneous violations, is widely adopted in constrained online optimization~\cite{Yu2017Online,Xinlei2025}. This metric is stricter than the conventional long-term constraint-violation metric in~\cite{Yuan2022Tac}.
 As a standard constraint qualification,  Slater's condition postulates the existence of a strictly feasible interior point~\cite{Boyd2004}. Leveraging this feasibility margin, the authors of~\cite{Yu2017Online} improve the long-term constraint violation by replacing the original constraint functions with non-negative clipped versions. However, as pointed out in~\cite{Xinlei2025}, this clipping operation destroys the strict feasibility margin. While the authors of~\cite{Xinlei2025} successfully tighten the bound in distributed single-level convex optimization by directly maximizing the regularized Lagrangian, extending this advancement to online bilevel optimization is nontrivial. Under the bilevel nested structure, the primal--dual updates dynamically couple with the continuous drift of the inner-level optimal trajectories. Most existing studies on online bilevel optimization such as~\cite{OnlienBI11,OnlineBereNips33,OnlineFirAcx} are confined to unconstrained settings or simple feasible domains. It remains unclear whether sublinear cumulative constraint violation and corresponding bound reductions under Slater’s condition can be achieved in online bilevel optimization with time-varying inequality constraints.

As the scale of the network expands, the prohibitive communication overhead rapidly emerges as the primary bottleneck to the overall efficiency of the distributed algorithm. Communication compression serves as an effective approach to alleviate this bottleneck, which has been extensively studied in single-level optimization. For instance, the authors of~\cite{zhangjiaqi} develop an error compensation strategy to achieve linear convergence, while~\cite{Pushi} and~\cite{JOTA} establish a unified framework accommodating both relative and absolute compression errors. Furthermore, the authors of~\cite{Cao2023a} propose a distributed online gradient descent algorithm with communication compression  attaining an $O(\sqrt{T})$ regret bound. \cite{YuanTnnls2026} and \cite{KunTac2026} respectively extend compression mechanisms to the push-sum protocol and the bandit feedback setting. However, directly extending these compression schemes to online bilevel optimization is nontrivial. Since both the inner-level solution and the Hessian-inverse-vector product are tracked based on the current outer-level iterate, perturbations induced by compression in the outer-level iterate also affect these tracking processes. Consequently, the resulting compression errors propagate into the hypergradient approximation rather than remaining only in the consensus dynamics. It remains unclear how to design communication-efficient online bilevel optimization algorithms with communication compression.

\begin{table*}[htbp]
	\renewcommand{\arraystretch}{1.5} % 调整表格行高
	\caption{COMPARISON WITH RELATED CENTRALIZED ONLINE BILEVEL ALGORITHMS}
	\label{tab:comparison_centralized}
	\centering
	\resizebox{\textwidth}{!}{ 
		\begin{tabular}{c|c|c|c|c|c|c}
			\hline
			Reference & \makecell{Problem \\ type} & \makecell{Gradient\\ Queries} & \makecell{HvP \\ Complexity} & \makecell{Single-loop} & \makecell{Dynamic Regret Bound} & \makecell{Cumulative constraint \\ violation} \\ \hline
			OBBO ($w=1$)~\cite{OnlineBereNips33} & Centralized & $\mathcal{O}(T \log T)$ & $\mathcal{O}(T \log T)$ & No & $\mathcal{O}(1+V_T+H_{2,T})$ & Unconstrained \\ \hline
			SOBOW ($w=1$)\cite{Linetal2023}& Centralized & $\mathcal{O}(T)$ & $\mathcal{O}(T^2)$ & No & $\mathcal{O}(1+V_T+H_{2,T})$ & Unconstrained \\ \hline
			FSOBO~\cite{Onlinelocalregret} & Centralized & $\mathcal{O}(T)$ & $\mathcal{O}(T)$ & Yes & $\mathcal{O}(1+V_T+H_{2,T}+E_{2,T})$ & Unconstrained\\ \hline
			AOBO~\cite{Onlinelocalregret} & Centralized & $\mathcal{O}(T \log T
			+H_{2,T})$ & $\mathcal{O}(T \log T)$ & No & $\mathcal{O}(1 + V_T)$ & Unconstrained\\ \hline
			Algorithm~\ref{algo1} & Distributed & $\mathcal{O}(T)$ & $\mathcal{O}(T)$ & Yes & $\mathcal{O}(T^{\max\{\theta_1,1-\theta_1\}}+T^{\theta_1 } P_T+H_{1,T}+E_{1,T})$ & \makecell{$\mathcal{O}(T^{1-\theta_1/2})$ (w/o Slater$^{\dagger}$) \\ $\mathcal{O}(T^{1-\theta_1})$ (w/ Slater$^{\ddagger}$)} \\ \hline
			Algorithm~\ref{alg.2} & Distributed & $\mathcal{O}(T\log(\sqrt{T})+H_{1,T})$ & $\mathcal{O}(T \log T)$ & No & $\mathcal{O}(T^{\max\{\theta_1,1-\theta_1\}}+T^{\theta_1 } P_T)$ & \makecell{$\mathcal{O}(T^{1-\theta_1/2})$ (w/o Slater$^{\dagger}$) \\ $\mathcal{O}(T^{1-\theta_1})$ (w/ Slater$^{\ddagger}$)} \\ \hline
		\end{tabular}
	}
	
	% ================== 表注 ==================
	\vspace{4pt} 
	\raggedright
	\textbf{Notes:} 
	Gradient Queries and HvP Complexity denote the number of first-order gradient evaluations and HvP approximations required over the $T$-round horizon, respectively. 
	$^{\dagger}$Bound obtained without assuming Slater's condition. $^{\ddagger}$Bound obtained under Slater's condition. Here, $w$ denotes the window size, and $\theta_1 \in (0,1)$ is a tunable parameter. For $\varXi  \in \{1, 2\}$, $H_{\varXi ,T}$ denotes the inner-level path variation defined in (\ref{HH}); $V_T$ represents the functional environmental drift given by $V_T = \sum_{t=1}^{T-1} \sup_{x \in \mathcal{X}} |f_{t+1}(x) - f_t(x)|$; $E_{\varXi ,T}$ is the gradient path variation defined in~(\ref{AA12}). $P_T$ is the outer-level path variation defined in (\ref{path}).
\end{table*}

\begin{table*}[htbp]
	\renewcommand{\arraystretch}{1.5} % 调整表格行高
	\caption{COMPARISON WITH RELATED DISTRIBUTED ONLINE ALGORITHMS}
	\label{tab:comparison_distributed}
	\centering
	% 直接控制字号：8pt 为字体大小，10pt 为基础行距（行距通常设为字号的 1.2 倍）
	\fontsize{7pt}{8pt}\selectfont 
	\begin{tabular}{c|c|c|c|c|c}
		\hline
		Reference & \makecell{Problem \\ type} & \makecell{Compressed \\ communication} & \makecell{Slater's \\ condition} & \makecell{Static Regret Bound} & \makecell{Cumulative constraint \\ violation} \\ \hline
		\multirow{2}{*}{\cite{Xinlei2025}} & \multirow{2}{*}{Single-level} & \multirow{2}{*}{No} & No & \multirow{2}{*}{$\mathcal{O}(T^{\max\{\theta_1,1-\theta_1\}})$} & $\mathcal{O}(T^{1-\theta_1/2})$ \\ \cline{4-4} \cline{6-6}
		&  &  & Yes &  & $\mathcal{O}(T^{1-\theta_1})$ \\ \hline
		\multirow{2}{*}{\cite{kunCdc}} & \multirow{2}{*}{Single-level} & \multirow{2}{*}{Yes} & No & \multirow{2}{*}{$\mathcal{O}(T^{\max\{\theta_1,1-\theta_1\}})$} & $\mathcal{O}(T^{1-\theta_1/2})$ \\ \cline{4-4} \cline{6-6}
		&  &  & Yes &  & $\mathcal{O}(T^{1-\theta_1})$ \\ \hline
		\multirow{2}{*}{Algorithm~\ref{algo1}} & \multirow{2}{*}{Bilevel} & \multirow{2}{*}{Yes} & No & \multirow{2}{*}{\makecell{$\mathcal{O}(T^{\max\{\theta_1,1-\theta_1\}} +H_{1,T}+E_{1,T})$}} & $\mathcal{O}(T^{1-\theta_1/2})$ \\ \cline{4-4} \cline{6-6}
		&  &  & Yes &  & $\mathcal{O}(T^{1-\theta_1})$ \\ \hline
		\multirow{2}{*}{Algorithm~\ref{alg.2}} & \multirow{2}{*}{Bilevel} & \multirow{2}{*}{Yes} & No & \multirow{2}{*}{$\mathcal{O}(T^{\max\{\theta_1,1-\theta_1\}})$} & $\mathcal{O}(T^{1-\theta_1/2})$ \\ \cline{4-4} \cline{6-6}
		&  &  & Yes &  & $\mathcal{O}(T^{1-\theta_1})$ \\ \hline
	\end{tabular}
\end{table*}

In this paper, we consider the distributed online  bilevel optimization problem with a convex outer-level objective, strongly convex inner-level problems, and time-varying inequality constraints. The main contributions are outlined as follows.
\begin{enumerate}[(i)]
	\item   We propose two distributed online primal--dual algorithms with communication compression by combining inner-level solution tracking with iterative HvP approximation. Different from existing distributed bilevel algorithms~\cite{chen2023decentralized111,DisBilAuto2023penalty,niu2025lopa,niu2026hessian}, which are primarily developed for static environments, and centralized online bilevel algorithms~\cite{OnlienBI11,Onlinelocalregret,OnlineBereNips33,OnlineFirAcx}, our framework simultaneously accounts for distributed implementation, online variations, time-varying inequality constraints, and communication compression. In particular, the first algorithm retains a single-loop architecture with one inner-level update and one HvP tracking step per round, while the second algorithm utilizes an adaptive multi-loop tracking mechanism to eliminate the regret dependence on inner-level path variation. Compared with centralized online bilevel methods~\cite{OnlineBereNips33,Onlinelocalregret,Linetal2023}, the proposed framework additionally accounts for network disagreement, compressed communication, and time-varying inequality constraints. Compared with distributed constrained online methods~\cite{Xinlei2025,kunCdc}, the bilevel structure further requires tracking both the time-varying inner solution and the HvP approximation, whose errors are coupled with the distributed primal dynamics.

	\item We show in Theorem~\ref{thm1} that the single-loop algorithm establishes an $\mathcal{O}(T^{\max\{\theta_1,1-\theta_1\}}+T^{\theta_1 } P_T+H_{1,T}+E_{1,T})$ dynamic network regret bound and an $\mathcal{O}(T^{1-\theta_1/2})$ network cumulative constraint violation bound, which is dependent on the outer-level path variation $P_T$, the inner-level path variation $H_{1,T}$ and the gradient path variation $E_{1,T}$.  Moreover, the network cumulative constraint violation bound is further reduced to $\mathcal{O}(T^{1-\theta_1})$ under Slater's condition.  The tightening is achieved by retaining the original constraints in the regularized Lagrangian, thereby preserving the negative feasibility margin supplied by the common Slater point. To the best of our knowledge, this is the first work to achieve such an improvement for distributed online bilevel optimization.

	\item We show in Theorem~\ref{thm2} that the adaptive inner-loop algorithm establishes an improved $\mathcal{O}(T^{\max\{\theta_1,1-\theta_1\}}+T^{\theta_1 }P_T)$ dynamic network regret bound which depends solely on the outer-level path variation. The resulting regret bound no longer depends explicitly on the inner-level and gradient path variations \(H_{1,T}\) and \(E_{1,T}\).  Different from~\cite{OnlienBI11,OnlineBereNips33} that utilize a prefixed number of inner gradient steps, our adaptive algorithm dynamically adjusts the number of inner-level updates based on a gradient-norm tolerance condition.  Moreover, the proposed adaptive algorithm achieves an $\mathcal{O}(T^{1-\theta_1/2})$ network cumulative constraint violation bound, and this bound is reduced to $\mathcal{O}(T^{1-\theta_1})$ under Slater's condition.
\end{enumerate}

To highlight our theoretical advantages, we compare our results with centralized online bilevel algorithms and  distributed online algorithms in TABLE~\ref{tab:comparison_centralized} and  TABLE~\ref{tab:comparison_distributed}, respectively.

The remainder of this paper is organized as follows. Section~\ref{sec2} introduces preliminaries and formulates the problem. Section~\ref{sec3} proposes the single-loop distributed compressed online primal--dual algorithm.  Section~\ref{sec4} develops the adaptive multi-loop distributed compressed online primal--dual algorithm. Section~\ref{sec5} provides numerical simulations to verify the theoretical findings. Finally, Section~\ref{sec6} concludes this paper.

\textsl{Notations:} The sets of real numbers and $n$-dimensional real column vectors are denoted by $\mathbb{R}$ and $\mathbb{R}^n$, respectively. $\mathbb{R}_+^n$ stands for the set of $n$-dimensional nonnegative vectors. $\mathbb{N}_+$ denotes the set of all positive integers. Given a positive integer $N$, we define the index set $[N] = \{1, 2, \ldots, N\}$. For vectors $x_1, \ldots, x_N$, their vertical concatenation is defined as $\text{col}(x_1, \ldots, x_N) = [x_1^T, \ldots, x_N^T]^T$.  $\|\cdot\|$ represents the standard Euclidean norm for vectors. $\|\cdot\|_1$ denotes the $L_1$-norm.  For a closed set $\mathcal{X}$, $\mathcal{P}_{\mathcal{X}}(x)$ denotes the standard Euclidean projection of $x$ onto $\mathcal{X}$. $[\cdot]_+$ denotes the Euclidean projection onto the nonnegative orthant. $\otimes$ represents the Kronecker product. The notation $\mathbb{E}[\cdot]$ stands for the expectation operator. For a matrix $W$, $[W]_{ij}$ denotes the matrix entry in the $i$-th row and $j$-th column. For two scalar sequences $\{\alpha_t\}$ and $\{\beta_t\}$, the notation $\alpha_t = \mathcal{O}(\beta_t)$ indicates the existence of a constant $a > 0$ such that $\alpha_t \leq a \beta_t$ holds for all $t$. $\mathcal{U}(a, b)$ denotes the uniform distribution over the interval $[a, b]$ with $a < b \in \mathbb{R}$. For any scalar $a \in \mathbb{R}$, $\lfloor a \rfloor$ denotes the floor function, outputting the greatest integer less than or equal to $a$.   $\mathbf{I}$ is the identity matrix and $\mathbf{1}_{p \times q}$ denotes the matrix of all ones. $\mathbf{0}_n$ and $\mathbf{1}_n$ denote the $n$-dimensional zero vector and all-one vector, respectively.

\section{Preliminaries and Problem Formulation}\label{sec2}
 In this section, we first introduce essential preliminary concepts regarding algebraic graph theory and Bregman divergence. Subsequently, we formalize the problem of distributed online  bilevel optimization subject to time-varying inequality constraints, alongside the corresponding standard assumptions.
 
\subsection{Graph Theory}
The communication topology among the $N$ agents is modeled by a directed graph $\mathcal{G}=([N], \mathcal{E})$, where $\mathcal{E}\subseteq[N]\times[N]$ denotes the set of edges. 
An edge $(j, i) \in \mathcal{E}$ indicates that agent $i$ can receive information from agent $j$. The in-neighbor set of agent  $i$ is defined as  $\mathcal{N}_i^{\text{in}}(\mathcal{G}) = \{j \in [N] \mid (j, i) \in \mathcal{E}\}$.  
A directed path from node $i_1$ to node $i_k$ is defined as a sequence of edges $(i_1, i_2), (i_2, i_3), \ldots, (i_{k-1}, i_k)$ belonging to $\mathcal{E}$. 
Let $W \in \mathbb{R}^{N \times N}$ be the mixing matrix associated with $\mathcal{G}$, where the element $[W]_{ij}>0$ if $(j, i) \in \mathcal{E}$ or $i=j$, and $[W]_{ij}=0$ otherwise. 
The matrix $W$ is said to be doubly stochastic if $W\mathbf{1}_N=W^T\mathbf{1}_N=\mathbf{1}_N$.  
\begin{assumption}\label{ass:network_fixed}
	The communication graph $\mathcal{G} = ([N], \mathcal{E})$ and its associated mixing matrix $W \in \mathbb{R}^{N \times N}$ satisfy the following conditions:
	\begin{enumerate}
		\item[(i)] The directed graph $\mathcal{G}$ is strongly connected.
		\item[(ii)] The mixing matrix $W$ is doubly stochastic.
		\item[(iii)] There exists a constant $\underline{w} \in (0, 1)$ such that $[W]_{ii} \ge \underline{w}$ for all $i \in [N]$, and $[W]_{ij} \ge \underline{w}$ for all $(j, i) \in \mathcal{E}$.
	\end{enumerate}
\end{assumption}

\subsection{Bregman Divergence}
For the sake of generality, this paper uses the Bregman divergence~\cite{Xinlei2025} to measure the distance of two points $x, y$ in a nonempty closed convex set $\mathcal{X} \subseteq \mathbb{R}^q$, which is defined as
$$
\textstyle  \mathcal{D}_{\psi}(x, y) = \psi(x) - \psi(y) - \langle \nabla \psi(y), x - y \rangle,
$$
where $\psi: \mathcal{X} \to \mathbb{R}$ is a continuously differentiable and $\eta$-strongly convex distance-generating function with constant $\eta > 0$. 

By construction, the function $\mathcal{D}_{\psi}(\cdot, y)$ inherits the $\eta$-strong convexity of $\psi$ for any given $y \in \mathcal{X}$. Consequently, the divergence satisfies the following uniform lower bound
\begin{equation}\label{strong}
\textstyle  \mathcal{D}_{\psi}(x, y) \geq \frac{\eta}{2} \|x - y\|^2, \quad \forall x, y \in \mathcal{X}.
\end{equation}

A fundamental property leveraged in our subsequent analysis is the three-point identity. Specifically, for any $x, y, z \in \mathcal{X}$, the following equality holds
\begin{equation}\label{three}
\textstyle  \langle y - x, \nabla \psi(z) - \nabla \psi(y) \rangle = \mathcal{D}_{\psi}(x, z) - \mathcal{D}_{\psi}(x, y) - \mathcal{D}_{\psi}(y, z).
\end{equation}

\subsection{Problem Formulation}
Consider a network of $N$ agents. Given a time horizon $T \in \mathbb{N}_+$,  at each round $t \in [T]$, the agents collaborate to minimize a global time-varying outer-level function $f_t(x)$ subject to local time-varying inequality constraints. In this hierarchical paradigm, $f_t(x)$ depends on the optimal solutions of personalized inner-level problems. Formally, the problem is formulated as
\begin{align}\label{Prob}
	\textstyle  \min_{x \in \mathcal{X}}  \; &  \textstyle  f_t(x) = \frac{1}{N} \sum_{i=1}^{N} f_{i,t}\big(x, y_{i,t}^*(x)\big) \nonumber\\
	\text{s.t.} \; &\textstyle   y_{i,t}^*(x) \in \arg\min_{y \in \mathbb{R}^p} g_{i,t}(x, y), \\
	&\textstyle   h_{i,t}(x) \leq \mathbf{0}_{m_i}, \quad \forall i \in [N], \nonumber
\end{align}
where $x \in \mathcal{X} \subset \mathbb{R}^q$ is the global decision variable constrained to a common compact convex set $\mathcal{X}$, and $y \in \mathbb{R}^p$ is the inner-level variable. For each agent $i$, $f_{i,t}: \mathbb{R}^q \times \mathbb{R}^p \rightarrow \mathbb{R}$ is the outer-level cost, $g_{i,t}: \mathbb{R}^q \times \mathbb{R}^p \rightarrow \mathbb{R}$ is the inner-level cost, and $h_{i,t}: \mathbb{R}^q \rightarrow \mathbb{R}^{m_i}$ represents the local inequality constraint. We denote the stacked constraint vector as $h_t(x) = \text{col}(h_{1,t}(x), \dots, h_{N,t}(x)) \in \mathbb{R}^m$, where $m = \sum_{i=1}^N m_i$. For each $t \in [T]$, define the instantaneous feasible set $\mathcal{X}_t = \{ x \in \mathcal{X} \mid h_t(x) \le 0_m \}$, and define $\Omega = \bigcap_{t=1}^T \mathcal{X}_t = \{ x \in \mathcal{X} \mid h_t(x) \le 0_m, \forall t \in [T] \}$. The set $\Omega$ is assumed to be nonempty.

In the considered online setting, at the beginning of each round $t$, agent $i$ selects a local decision $x_{i,t}\in\mathcal X$ using only the information available up to round $t-1$. After the decision is made, the current local functions $f_{i,t}$, $g_{i,t}$, and $h_{i,t}$, together with the derivative information become available to agent $i$.

A representative application of (\ref{Prob}) is hierarchical resource allocation over communication networks. In a convex--strongly-convex bilevel formulation, the lower level determines users' resource or transmission-rate allocations under a given upper-level control parameter, whereas the outer level adjusts this parameter to minimize the true aggregate network cost~\cite{ji2023lower}. In an online distributed environment, local traffic loads, channel conditions, and capacity limits may vary over time, which can be modeled by the time-varying functions $f_{i,t}$, $g_{i,t}$, and constraints $h_{i,t}$ in (\ref{Prob}).

The primary goal of agents is to design the decision sequences $\{x_{i,t}\}$ for $i \in [N]$ and $t \in [T]$ to ensure sublinear growth of the network regret and the network cumulative constraint violation. Specifically, the dynamic network regret is formally defined as
\begin{align}
\textstyle 	\text{Net-Reg}(T) = \frac{1}{N} \sum_{i=1}^{N} \sum_{t=1}^{T} \left[ f_t(x_{i,t}) - f_t(x_t^*) \right],
\end{align}
where $x_t^* \in \arg\min_{x \in \mathcal{X}_t} {1}/{N} \sum_{i=1}^N f_{i,t} (x, y_{i,t}^*(x))$ denotes an optimal solution of problem (\ref{Prob}) at time  $t$. This metric compares the local decisions generated by the network with the sequence of instantaneous constrained optimal solutions $\{x_t^*\}_{t=1}^T$. Hence, it quantifies the ability of the distributed algorithm to track the time-varying outer-level optimum.

Furthermore,  we also evaluate the static network regret. Unlike its dynamic counterpart, static regret measures the performance gap against the best fixed decision in hindsight. 
Let $x^* \in \arg\min_{x \in \Omega} {1}/{N} \sum_{t=1}^T \sum_{i=1}^N f_{i,t} (x, y_{i,t}^*(x))$. The static network regret is defined as
\begin{align}
	&\textstyle  \text{Net-SReg}(T) = \frac{1}{N} \sum_{i=1}^{N} \sum_{t=1}^{T} \left[ f_t(x_{i,t}) - f_t(x^*) \right],
\end{align}
 which compares the generated decisions with the best fixed feasible decision over the entire horizon.

Additionally, the network cumulative constraint violation is given by
\begin{equation}
\textstyle 	\text{Net-CCV}(T) = \frac{1}{N} \sum_{i=1}^N \sum_{t=1}^T \big\| [h_t(x_{i,t})]_+ \big\|.
\end{equation}
This metric evaluates each agent's local decision against the global constraint  $h_t$, independently penalizing instantaneous violations at each round.

The nested structure of (\ref{Prob}) couples the outer-level objective with the inner-level optimal solution, which complicates gradient evaluation. If $g_{i,t}(x,\cdot)$ is $\mu_g$-strongly convex and twice continuously differentiable, then $y_{i,t}^*(x)$ is unique and
	$$ \textstyle  \nabla_{yy}^2g_{i,t}\bigl(x,y_{i,t}^*(x)\bigr) \succeq \mu_g I, $$
	so the inner-level Hessian is nonsingular. Applying the implicit function theorem~\cite{niu2025lopa} to the first-order optimality condition
	$$ \textstyle  \nabla_y g_{i,t}\bigl(x,y_{i,t}^*(x)\bigr)=0, $$
	the exact local hypergradient  of the composite outer-level function with respect to $x$ can be written as
	\begin{align}\label{eq:hypergradient}
		&\textstyle  \nabla f_{i,t}(x_{i,t},y_{i,t}^*(x_{i,t})) = \nabla_x f_{i,t}(x_{i,t}, y_{i,t}^*(x_{i,t})) \nonumber \\
		&\qquad \textstyle  + \nabla y_{i,t}^*(x_{i,t}) \nabla_y f_{i,t}(x_{i,t}, y_{i,t}^*(x_{i,t})), 
	\end{align}
	where the matrix $\nabla y_{i,t}^*(x_{i,t})$ is given by
	\begin{equation}
		\begin{aligned}\label{y_star}
			&\textstyle  \nabla y_{i,t}^*(x_{i,t}) \\
			&= \textstyle  -\nabla_{xy}^2 g_{i,t}(x_{i,t}, y_{i,t}^*(x_{i,t})) \left[ \nabla_{yy}^2 g_{i,t}(x_{i,t}, y_{i,t}^*(x_{i,t})) \right]^{-1}.
		\end{aligned}
\end{equation}
This composite structure explicitly reveals a computational bottleneck evaluating $\nabla f_{i,t}(x_{i,t}, y_{i,t}^*)$ requires acquiring the exact inner-level optimum $y_{i,t}^*(x_{i,t})$ and explicitly inverting the high-dimensional inner-level Hessian inverse matrix $\left[ \nabla_{yy}^2 g_{i,t}(x_{i,t}, y_{i,t}^*(x_{i,t})) \right]^{-1}$.

To characterize the temporal non-stationarity that enters the regret analysis, similar to~\cite{Onlinelocalregret,OnlineFirAcx}, we introduce variation measures for the outer-level optimal solution, the inner-level optimal solution, and the gradients involved in the Hessian-inverse-vector product (HvP) approximation. Specifically, the outer-level path variation is defined as
\begin{equation}\label{path}
	\textstyle P_T = \sum_{t=1}^{T-1} \|x_{t+1}^* - x_t^*\|. 
\end{equation}

For $ \Xi \in \{1, 2\}$, the inner-level path variation is defined as
\begin{equation}\label{HH}
\textstyle 	H_{\varXi ,T} = \sum_{t=1}^{T-1} \sup_{x \in \mathcal{X}, i \in [N]} \|y_{i,t+1}^*(x) - y_{i,t}^*(x)\|^\varXi.
\end{equation}

To circumvent the exact Hessian inverse matrices, our algorithms need to track the HvP $ \left[ \nabla_{yy}^2 g_{i,t}(x_{i,t}, y_{i,t}^*(x_{i,t})) \right]^{-1} \nabla_y f_{i,t}(x_{i,t}, y_{i,t}^*(x_{i,t}))$. To bound the tracking error, we define the gradient path variation 
\begin{align}\label{AA12}
\textstyle  E_{\varXi ,T} = E_{\varXi ,yy,T}^g + E_{\varXi ,y,T}^f, \quad \varXi \in \{1,2\},
\end{align}
 where
\begin{align*}
&\textstyle E_{\varXi ,yy,T}^g \nonumber\\
&\textstyle  = \sum_{t=1}^{T-1} \sup_{\substack{x \in \mathcal{X}, y \in \mathbb{R}^p \\ i \in [N]}} \|\nabla_{yy}^2 g_{i,t+1}(x,y) - \nabla_{yy}^2 g_{i,t}(x,y)\|^\varXi , \\
&\textstyle E_{\varXi ,y,T}^f \nonumber\\
&\textstyle  = \sum_{t=1}^{T-1} \sup_{\substack{x \in \mathcal{X}, y \in \mathbb{R}^p \\ i \in [N]}} \|\nabla_y f_{i,t+1}(x,y) - \nabla_y f_{i,t}(x,y)\|^\varXi .
\end{align*}

To proceed, we impose the following standard assumptions, which are mild and widely adopted in the literature~\cite{niu2025lopa,Xinlei2025,kunCdc}.
\begin{assumption} \label{ass:outer}
	For any $t \in \mathbb{N}_+$, the following conditions hold 
	\begin{enumerate}
		\item[(i)] For each $i\in [N]$, $f_{i,t}(x,y)$ is jointly continuously differentiable with respect to $(x,y)$.
		\item[(ii)] $\nabla_x f_{i,t}(x,y)$ and $\nabla_y f_{i,t}(x,y)$ are uniformly bounded by positive constants $C_{f,x}$ and $C_{f,y}$, respectively.
		\item[(iii)] $\nabla_x f_{i,t}(x,y)$ and $\nabla_y f_{i,t}(x,y)$ are Lipschitz continuous with respect to $(x,y)$ with positive constants $L_{f,x}$ and $L_{f,y}$, respectively.
			\item[(iv)] The outer-level function  $f_t : \mathcal{X} \to \mathbb{R} $ defined in (\ref{Prob}) is convex on $ \mathcal{X}$.
	\end{enumerate}
\end{assumption}

\begin{assumption}\label{ass:inner}
	For each $i \in [N]$ and any $t \in \mathbb{N}_+$, the inner-level cost function $g_{i,t}: \mathbb{R}^q \times \mathbb{R}^p \to \mathbb{R}$ satisfies
	\begin{enumerate}
		\item[(i)] $g_{i,t}(x, y)$ is twice continuously differentiable with respect to $(x,y)$, and for any fixed $ x \in \mathbb{R}^q, g_{i,t}(x, \cdot)$  is $\mu_g$-strongly convex with respect to $y$.
		\item[(ii)] The gradient $\nabla_y g_{i,t}(x,y)$ is Lipschitz continuous with respect to $y$ with a positive constant $L_{g,y}$. Furthermore,  $\nabla_{xy}^2 g_{i,t}(x,y)$ and $\nabla_{yy}^2 g_{i,t}(x,y)$ are Lipschitz continuous in $(x,y)$ with positive constants $L_{g,xy}$ and $L_{g,yy}$, respectively.
		\item[(iii)] $\nabla_{xy}^2 g_{i,t}(x,y)$ is uniformly bounded by a positive constant $C_{g,xy}$.
	\end{enumerate}
\end{assumption}

\begin{assumption}\label{ass:constraint}
	For each $i \in [N]$ and any $t \in \mathbb{N}_+$, the local constraint function $h_{i,t}: \mathbb{R}^q \to \mathbb{R}^{m_i}$ satisfies
	\begin{enumerate}
		\item[(i)] Each component $[h_{i,t}(x)]_l$, $l \in [m_i]$, is convex over $\mathbb{R}^q$.
		\item[(ii)] $\nabla h_{i,t}(x)$ is uniformly bounded by a positive constant $G_h$.
	\end{enumerate}
\end{assumption}

\begin{assumption}\label{ass.mir}
For any fixed $x \in \mathcal{X}$, the Bregman divergence $\mathcal{D}_{\psi}(x, \cdot): \mathcal{X} \to \mathbb{R}$ is convex with respect to its second argument. And there exists a constant  $L_{\psi} > 0$ such that, $\| \nabla \psi(x_1) - \nabla \psi(x_2) \| \leqslant L_{\psi} \| x_1 - x_2 \|$, $\forall x_1, x_2 \in \mathcal{X}$. Furthermore,  there exists a constant $K > 0$ such that
	\begin{equation}\label{Klip}
		\textstyle |\mathcal{D}_{\psi}(x_1, y) - \mathcal{D}_{\psi}(x_2, y)| \leq K \|x_1 - x_2\|, \forall x_1, x_2, y \in \mathcal{X}.
	\end{equation}
\end{assumption}

Since $\mathcal{X}$ is compact, its diameter $d(\mathcal{X}) = \sup_{x,y \in \mathcal{X}} \Vert{}x-y\Vert{}$ is finite. By (\ref{Klip}) and $\mathcal{D}_{\psi}(y, y)=0$, it holds that
\begin{equation}\label{KKbound}
	\mathcal{D}_{\psi}(x, y) \leq K d(\mathcal{X}), \quad \forall x, y \in \mathcal{X}.
\end{equation}

\begin{assumption}\label{ass.slater}
	There exists a strictly feasible point $\bar{x} \in \mathcal{X}$ and a constant $\varsigma_s > 0$, such that for each agent $i \in [N]$ and all $t \in [T]$, it holds
	\begin{equation}
	\textstyle 	h_{i,t}(\bar{x}) \le -\varsigma_s \mathbf{1}_{m_i}.
	\end{equation}
\end{assumption}

To reduce the communication overhead, we introduce the  communication compression mechanism. By broadcasting compressed and  lower-bit representations instead of exact real-valued vectors, the communication burden can be significantly alleviated. We adopt a class of compression operators, which is characterized by the following assumption.
\begin{assumption}\label{ass.cpmpress}
	The compressor $\mathcal{C}:\mathbb{R}^{q}\rightarrow \mathbb{R}^{q}$ satisfies
	\begin{equation}
	\textstyle 	\mathbb{E}_{\mathcal{C}}\left[\left\|{\mathcal{C}(x)}-x\right\|^{2}\right]\le C,\quad \forall x\in\mathbb{R}^{q}, \label{eq:compressor} 
	\end{equation}
	where constant $C\ge0$ and  $\mathbb{E}_{\mathcal{C}}[\cdot]$ denotes the expectation with respect to the inherent randomness of the compressor. 
\end{assumption}

\begin{remark}
The problem formulation in (\ref{Prob}) offers a unified framework that encapsulates conventional optimization settings. Notably, it reduces to standard online constrained optimization~\cite{Cao2023a,Xinlei2025,kunCdc} when $f_{i,t}(x,y)$ is independent of $y$, and reduces to an online minimax problem~\cite{WentaoTac,WentaoAuto,OnlineSa} when $f_{i,t}(x,y)=-g_{i,t}(x,y)$. 
Evaluating the exact hypergradient in (\ref{Prob}) is a general computational bottleneck in bilevel optimization, since it requires both the exact inner-level solution \(y_{i,t}^\star(x)\) and the Hessian inverse $[\nabla_{yy}^2 g_{i,t}(x,y)]^{-1}$. In the distributed online setting considered here, an additional difficulty arises because each agent evaluates its hypergradient at a local time-varying iterate \(x_{i,t}\). Consequently, network disagreement and compression errors perturb the evaluation points of both the inner-solution and HvP trackers, and these perturbations further propagate into the outer-level primal--dual dynamics.
\end{remark} %%%%%%%%%%%%%%%%%%%%%%%%%%%%%%%%%

\section{Single-loop Distributed Compressed Online Primal--Dual algorithm}\label{sec3}
In this section, the Single-loop Distributed Compressed Online Primal--Dual (SDCO-PD) algorithm is designed for online bilevel optimization problem~(\ref{Prob}). Moreover, the performance of the algorithm is analyzed without and with Slater's condition, respectively.

\subsection{Algorithm Description}
To motivate the distributed design, we consider a saddle-point formulation. Recall that at the $t$-th round, the global cost function and constraint function $h_t$ are given in (\ref{Prob}). The associated regularized Lagrangian function is defined as
\begin{equation}
	\textstyle \mathcal{L}_t(x_t, \varphi_t) = \frac{1}{N} \sum_{i=1}^N f_{i,t}(x_t, y_{i,t}^*) + \varphi_t^T h_t(x_t) - \frac{1}{2\gamma_t} \|\varphi_t\|^2, \label{eq:lagrangian}
\end{equation}
where $\varphi_t \in \mathbb{R}_+^m$ is the Lagrange multiplier associated with the time-varying inequality constraint and $\gamma_t$ is the regularization parameter.  In a centralized setting, an online primal--dual projected gradient algorithm can be given by
\begin{equation}\label{Sanjia}
	\begin{cases}
		x_{t+1} = P_{\mathcal{X}}(x_t - \alpha_t \nabla_x L_t(x_t, \varphi_t)) \\
		\varphi_{t+1} = [\varphi_t + \gamma_t \nabla_\varphi L_t(x_t, \varphi_t)]_+,
	\end{cases} 
\end{equation}
where $\alpha_t > 0$ is the primal step size and $\nabla_x L_t(x_t, \varphi_t) = {1}/{N} \sum_{i=1}^N \nabla f_{i,t}(x_{t},y_{i,t}^*) + \nabla h_t(x_t)^T \varphi_t$. However, evaluating the exact global hypergradient requires the aggregation of all agents' local information, which is inaccessible locally. It is imperative to extend the centralized algorithm to a distributed setting.

Inspired by~\cite{Xinlei2025}, by directly maximizing the regularized Lagrangian, each agent updates its local dual variable as
\begin{equation}
	\varphi_{i,t + 1} = \gamma _t [h_{i,t}( x_{i,t} )]_ +. \label{Algorithm1-eq4}
\end{equation}
 With the updated dual variable, the pre-consensus variable $z_{i,t+1}$ as a local estimate of $x_{t+1}$ is assigned to agent $i$. Instead of using the projected gradient descent in (\ref{Sanjia}), $z_{i,t+1}$ is updated via a mirror descent step
\begin{equation}\label{Algorithm1-eq6}
	z_{i,t+1} = \arg\min_{x \in \mathcal{X}} \left\{ \alpha_t \langle x, \omega_{i,t+1} \rangle + \mathcal{D}_{\psi}(x, x_{i,t}) \right\}, 
\end{equation}
where
\begin{equation}
	{{\omega} _{i,t + 1}} = \tilde{\nabla}f_{i,t}(x_{i,t},y_{i,t+1},v_{i,t+1}) + {\nabla} {{h_{i,t}}({x_{i,t}}) }^{T}{\varphi_{i,t+1}}. \label{Algorithm1-eq5}
\end{equation}
Here, $\tilde{\nabla}f_{i,t}(x_{i,t},y_{i,t+1},v_{i,t+1})$ is an approximated local hypergradient that will be detailed shortly to resolve the aforementioned hypergradient evaluation challenge.

To reduce the communication overhead, we employ a communication compression mechanism. Agent $i$ maintains the auxiliary variable $\hat{z}_{i,t}$, which serves as the local estimate of its own previously compressed pre-consensus variable $z_{i,t}$. At time $t$, agent $j$ compresses the scaled innovation between $z_{j,t}$ and $\hat{z}_{j,t-1}$ using the compressor $\mathcal{C}(\cdot)$ prior to broadcasting. Upon receiving the compressed messages from neighbors, agent $i$ updates the auxiliary variables via
\begin{flalign}\label{Algorithm1-eq1}
\textstyle 	{\hat{z}_{j,t}} = & \textstyle  \mathcal{P}_{\mathcal{X}}( {\hat{z}_{j,t - 1}} + {s_t}\mathcal{C}\big( {( {{z_{j,t}} - {\hat{z}_{j,t-1}}} )/{s_{t}}} \big)), \nonumber\\
	&\textstyle  j \in \big\{\mathcal{N}_i^{\text{in}}( {{\mathcal{G}}} ) \cup \{i\}\big\}. 
\end{flalign}

Subsequently, the local primal decision variable $x_{i,t}$ is constructed by 
\begin{align}\label{Algorithm1-eq3}
\textstyle  	{x_{i,t}} = \sum_{j = 1}^N {{{[ {{W}} ]}_{ij}}{{\hat z}_{j,t}}}. 
\end{align}

It should be noted that we employ a dynamic scaling technique to mitigate the absolute compression error in~(\ref{Algorithm1-eq1}). By Assumption~\ref{ass.cpmpress} and the non-expansiveness of the Euclidean projection, the projected compression error remains bounded by $s_t^2C$ in mean square. As $s_t$ decays to zero, the compression error therefore vanishes accordingly.

Bilevel optimization differs from single-level optimization in that outer-level functions depend on the optimal solutions of inner-level problems. 
Directly implementing the exact hypergradient $\nabla f_{i,t}(x_{i,t},y_{i,t}^*)$ given in \eqref{eq:hypergradient} is computationally prohibitive because evaluating the hypergradient requires the exact inner-level solution $y_{i,t}^*(x_{i,t})$ and the explicit Hessian inverse matrices, both necessitating expensive iterative subroutines.

To address this issue and alleviate the computation cost, we adopt two one-step approximation strategies to construct the surrogate hypergradient $\tilde{\nabla}f_{i,t}(x_{i,t},y_{i,t+1},v_{i,t+1})$. First,  a one-step gradient descent is performed to approximate the inner-level solution. We substitute the exact $y_{i,t}^*$ with a tracking variable $y_{i,t}$, reusing its previous value to warm-start the tracking for a better approximation
\begin{flalign}\label{estY}
	\textstyle  y_{i,t+1} = y_{i,t} - \beta \nabla_y g_{i,t}(x_{i,t}, y_{i,t}),
\end{flalign}
where $ \beta > 0 $ is the inner-level stepsize.

Second, we aim to approximate the Hessian inverse matrices. Define  $v_{i,t}(x,y) = [\nabla_{yy}^2 g_{i,t}(x,y)]^{-1} \nabla_y f_{i,t}(x,y)$, and let  $v_{i,t}^*(x) = v_{i,t}(x, y_{i,t}^*(x))$. The exact HvP required in (\ref{eq:hypergradient}) is therefore $v_{i,t}^*(x_{i,t})$, which  can be naturally formulated as the optimal solution to a strongly convex quadratic programming problem
\begin{equation}
	\label{eq:qp_problem}
	\begin{split}
		\textstyle \arg\min_{{v}} \Big\{ &\textstyle   \frac{1}{2} {v}^T \nabla_{yy}^2 g_{i,t}(x_{i,t}, y_{i,t}^*(x_{i,t})) {v} \\
		&\textstyle  - \nabla_y f_{i,t}(x_{i,t}, y_{i,t}^*(x_{i,t}))^T {v} \Big\}.
	\end{split}
\end{equation}
In order to avoid producing another computation loop when solving \eqref{eq:qp_problem}, we update the HvP tracking variable only once at each iteration, warm-started using its value at the previous iteration via a single gradient step
\begin{align}\label{eq:w_update}
	w_{i,t+1} = &~v_{i,t} - \lambda ( \nabla_{yy}^2 g_{i,t}(x_{i,t}, y_{i,t+1}) v_{i,t}\nonumber\\
	& \quad- \nabla_y f_{i,t}(x_{i,t}, y_{i,t+1}) ),
\end{align}
where $\lambda > 0$ is the HvP-tracking stepsize.

To maintain the boundedness of $w_{i,t}$, we employ an auxiliary projection onto the closed Euclidean ball with fixed radius $r_v \ge  C_{f,y}/\mu_g$ via
\begin{align}\label{eq:v_update}
	v_{i,t+1} &= 
	\begin{cases}
		w_{i,t+1}, & \|w_{i,t+1}\| \le r_v, \\
		\dfrac{r_v w_{i,t+1}}{\|w_{i,t+1}\|}, & \text{otherwise}.
	\end{cases}
\end{align}
This geometry-constrained projection fundamentally bounds the tracking magnitude, enforcing the uniform boundedness of the hypergradient. As a result, the approximated local hypergradient is  assembled as
\begin{align}\label{HuGradi}
	&\tilde{\nabla}f_{i,t}(x_{i,t},y_{i,t+1},v_{i,t+1}) = \nabla_x f_{i,t}(x_{i,t}, y_{i,t+1}) \nonumber \\
	&\qquad - \nabla_{xy}^2 g_{i,t}(x_{i,t}, y_{i,t+1}) v_{i,t+1}.
\end{align}
The main technical difficulty in the proposed framework arises from the interaction among the distributed primal dynamics, communication compression, inner level tracking, and HvP approximation. In particular, the distributed and compressed primal updates induce time varying variations in \(\|x_{i,t}-x_{i,t-1}\|\), which enter both the inner level tracking recursion and the HvP tracking recursion, as shown in (\ref{lemmYY}) and (\ref{eq:lem_v_bound_final}). These tracking errors further propagate into the hypergradient approximation and subsequently affect the outer level primal dual update. As a result, the consensus error, compression error, and bilevel approximation error are intrinsically coupled throughout the algorithmic dynamics. To address this difficulty, our analysis establishes a unified recursive error bound that jointly controls these coupled terms and leads to the dynamic regret and cumulative constraint violation guarantees of the proposed single loop algorithm.

By grouping the above updates, we get the SDCO-PD algorithm as summarized in  Algorithm~\ref{algo1}, which avoids nested computational loops and explicit matrix inversions, making it suitable for online distributed implementation.

\begin{algorithm}[!t] 
	\caption{Single-loop Distributed Compressed Online Primal--Dual (SDCO-PD) algorithm} 
	\begin{algorithmic}\label{algo1}
		\renewcommand{\algorithmicrequire}{\textbf{Input:}}
		\REQUIRE
		Non-increasing $\{ {\alpha _t}\},  \{ {{s _t}} \} \subseteq ( {0, + \infty })$, non-decreasing $\{ {\gamma _t}\} \subseteq ( {0, + \infty })$, step sizes $\beta, \lambda > 0$, and radius $r_v\ge C_{f,y}/\mu_g$.
		\renewcommand{\algorithmicrequire}{\textbf{Initialize:}}
		\REQUIRE
		${\hat{z}_{j,0}} = {\mathbf{0}_q}$ for $j \in [ N ]$, and ${z_{i,1}} \in {\mathcal{X}}$, $y_{i,1} = 0_p$, $v_{i,1} = 0_p$, $\varphi_{i,1} = 0_{m_i}$, $\forall i \in [ N ]$.
		\FOR {$t = 1, \cdot  \cdot  \cdot $}
		\FOR {$i = 1,\cdot  \cdot  \cdot, N$ in parallel}
		\STATE Update the projected compressed variable  ${\hat{z}_{j,t}}$ via \eqref{Algorithm1-eq1}.
		\STATE Select the local decision variable ${x_{i,t}}$ via \eqref{Algorithm1-eq3}.
		\STATE Observe the current local functions \(f_{i,t}\), \(g_{i,t}\), and \(h_{i,t}\), and access the required first- and second-order derivative oracles at the queried points. 
		\STATE Update the local dual variable ${\varphi_{i,t + 1}}$ via \eqref{Algorithm1-eq4}.
		\STATE Update the inner-level variable $y_{i,t+1}$ via \eqref{estY}.
		\STATE Update the HvP tracking variable $w_{i,t+1}$ via \eqref{eq:w_update}.
		\STATE Update the estimate $v_{i,t+1}$ via \eqref{eq:v_update}.
		\STATE Approximate the hypergradient $\tilde{\nabla}f_{i,t}(x_{i,t},y_{i,t+1},v_{i,t+1})$ via \eqref{HuGradi}.
		\STATE Update the gradient tracking direction ${{\omega} _{i,t + 1}}$ via \eqref{Algorithm1-eq5}.
		\STATE Update the pre-consensus primal variable $z_{i,t+1}$ via \eqref{Algorithm1-eq6}. 
		\ENDFOR
		\ENDFOR
		\renewcommand{\algorithmicensure}{\textbf{Output:}}
		\ENSURE
		$\{ x_{i,t} \}$.
	\end{algorithmic}
\end{algorithm}

\subsection{Performance Analysis}
In this section, we establish the theoretical guarantees of the SDCO-PD algorithm. 
\begin{theorem}\label{thm1}
	Suppose Assumptions~\ref{ass:network_fixed}--\ref{ass.mir} and \ref{ass.cpmpress} hold. For all $i \in [N]$, let $\{x_{i,t}\}$ be the sequences generated by Algorithm~\ref{algo1} with parameters $\beta \leq {2}/{(\mu_g + L_{g,y})}$, $\lambda \leq {1}/{L_{g,y}}$, and
	\begin{equation}\label{canshu}
		\alpha_t = \frac{\alpha_0}{t^{\theta_1}}, \quad \gamma_t = \frac{\gamma_0}{\alpha_t}, \quad s_t=\frac{s_0}{t^{\theta_2}}, \quad \forall t \in \mathbb{N}_+, 
	\end{equation}
	where $\alpha_0, s_0 > 0$, $\theta_1 \in (0, 1)$, $\theta_2=\theta_1+1$, and $\gamma_0 \in (0, \eta/(4G_h^2)]$ are constants. Then, for any $T \in \mathbb{N}_+$, the dynamic network regret and network cumulative constraint violation satisfy
	\begin{align}
		&\mathbb{E}[\text{Net-Reg}(T)] = \mathcal{O}(T^{\max\{\theta_1,1-\theta_1\}} + T^{\theta_1 } P_T+H_{1,T} + E_{1,T}), \label{thm1ret}  \\
		&\mathbb{E}[\text{Net-CCV}(T)] = \mathcal{O}(T^{1-\theta_1/2}). \label{thm1Cvv11}
	\end{align}
	Moreover, if Assumption~\ref{ass.slater} holds, then
	\begin{equation}
		\mathbb{E}[\text{Net-CCV}(T)] = \mathcal{O}(T^{1-\theta_1}).\label{thm1Cvv22} 
	\end{equation}
\end{theorem}
\begin{proof}
	The detailed proof is provided in Appendix~\ref{BBBB}.
\end{proof}

The variation terms \(H_{1,T}\) and \(E_{1,T}\) in (\ref{thm1ret}) accumulate temporal changes in norm, whereas the corresponding bounds in~\cite{Onlinelocalregret} involve the squared variations \(H_{2,T}\) and \(E_{2,T}\). Consequently, the contribution of a large environmental change grows linearly with its magnitude in our bound rather than quadratically. The norm based treatment also simplifies the HvP approximation analysis. 
In Lemma~\ref{lem:v_bound_dist_multi_norm}, splitting the error at \(\hat v_{i,t}=v_{i,t}(x_{i,t},y_{i,t+1})\), applying the contraction of the projected iteration in (\ref{eq:w_multi_loop}) and (\ref{eq:v_proj_loop}), and using the Lipschitz bound in (\ref{eq:lip_v}) yield (\ref{eq:v_bound_dist_res_norm}). 
The norm-level analysis nevertheless introduces a different coupling difficulty in the single-loop case. In particular, the HvP recursion in (\ref{eq:lem_v_bound_final}) contains the inner increment \(\|y_{i,t+1}-y_{i,t}\|\). Instead of treating this increment as an independent error, the update in (\ref{estY}) and the Lipschitz continuity of \(\nabla_y g_{i,t}\) give \(\|y_{i,t+1}-y_{i,t}\|\leq\beta L_{g,y}\|y_{i,t}-y_{i,t}^{*}(x_{i,t})\|\). The weight \(\kappa_F\) is then cleverly designed so that the contraction of the inner tracking error absorbs the additional coupling term induced by the HvP recursion. This yields the contractive recursion in~(\ref{AA98}) and preserves the linear dependence on \(H_{1,T}\) and \(E_{1,T}\) after summation.

\begin{corollary}\label{cor1}
Suppose the conditions of Theorem~\ref{thm1} hold. The static network regret and network cumulative constraint violation of Algorithm~\ref{algo1} satisfy
	\begin{align}
		\mathbb{E}[\text{Net-SReg}(T)] =& \mathcal{O}(T^{\max\{\theta_1,1-\theta_1\}} + H_{1,T} + E_{1,T})\\
		\mathbb{E}[\text{Net-CCV}(T)] =& \mathcal{O}(T^{1-\theta_1/2}).
	\end{align}
Furthermore, under Assumption~\ref{ass.slater}, it holds
	\begin{equation}
		\mathbb{E}[\text{Net-CCV}(T)] = \mathcal{O}(T^{1-\theta_1}). 
	\end{equation}
\end{corollary}
\begin{remark}
	In online constrained optimization, the cumulative constraint violation strictly penalizes instantaneous violations at each round, representing a stricter criterion than the conventional long-term constraint violation metric in~\cite{Yuan2022Tac}. As pointed out in~\cite{Xinlei2025}, while existing algorithms typically replace the original constraint function $h_t(x)$ with its non-negative clipped version $[h_t(x)]_+$ inside the Lagrangian, such an operation destroys the strict feasibility margin provided by Slater's condition. In contrast,  our proposed primal--dual updates directly maximize the regularized Lagrangian without early clipping. This design preserves the strict feasibility margin $\varsigma_s$, thereby successfully tightening the network cumulative constraint violation bound from $\mathcal{O}(T^{1-\theta_1/2})$ to $\mathcal{O}(T^{1-\theta_1})$. 
Compared with existing online bilevel optimization methods~\cite{OnlienBI11,OnlineBereNips33,OnlineFirAcx,Onlinelocalregret}, which do not provide constraint violation guarantees for time-varying inequality constraints, our result provides, to the best of our knowledge, the first such guarantee for online bilevel optimization. Compared with the Slater-based improvement established in single-level distributed online optimization~\cite{Xinlei2025,kunCdc}, the bilevel setting additionally requires controlling inner-level and HvP tracking errors together with network disagreement and compression errors while preserving the tightened violation bound.
\end{remark} %%%%%%%%%%%%%%%%%%%%%%%%%%%%%%%%%%%%%%%%%%%%%%%%%%%%

\section{Adaptive Multi-loop Distributed Compressed Online Primal--Dual Algorithm}\label{sec4}
In this section, an Adaptive Multi-loop Distributed Compressed Online Primal--dual (AMDCO-PD) algorithm is presented. Furthermore, its performance without and with Slater's condition is analyzed.

\subsection{Algorithm Description}
As shown in Theorem~\ref{thm1}, the dynamic network regret of Algorithm~\ref{algo1} depends on the inner-level and gradient path variation. Consequently, Algorithm~\ref{algo1} struggles to adapt to adversarial variations in the inner-level functions. To overcome this, we adopt a dynamic inner iteration strategy to track the drift of $y_{i,t}^*(x)$ at each round.

First, instead of a single gradient step, warm-started with $y_{i,t+1}^{(0)} = y_{i,t}$, the inner-level variable is updated via
\begin{equation}\label{eq:y_inner_loop}
	y_{i,t+1}^{(k+1)} = y_{i,t+1}^{(k)} - \beta \nabla_y g_{i,t}(x_{i,t}, y_{i,t+1}^{(k)}),
\end{equation}
where $k$ denotes the inner iteration index. Instead of a fixed number of steps, the loop adaptively terminates when the inner gradient norm satisfies $\|\nabla_y g_{i,t}(x_{i,t}, y_{i,t+1}^{(k)})\| \le \delta$  with $\delta>0$ being a prescribed inner-level gradient norm tolerance, and we set the final evaluation point as $y_{i,t+1} = y_{i,t+1}^{(k)}$. This adaptive stopping criterion explicitly controls the approximation error, tightly characterizing the exact inner optimum within a $\delta$-neighborhood.

Subsequently, to estimate the HvP without requiring high-dimensional matrix inversion, we perform a fixed $M$-step projected gradient descent. Let $m \in \{1, \dots, M\}$ denote the iteration index for the HvP estimation. Initialized with $v_{i,t+1}^{(0)} = v_{i,t}$, the tracking variable is updated via
\begin{align}\label{eq:w_multi_loop}
	w_{i,t+1}^{(m)} =& v_{i,t+1}^{(m-1)} - \lambda ( \nabla_{yy}^2 g_{i,t}(x_{i,t}, y_{i,t+1}) v_{i,t+1}^{(m-1)} \nonumber\\
	&\quad - \nabla_y f_{i,t}(x_{i,t}, y_{i,t+1}) ).
\end{align}
Similar to (\ref{eq:v_update}), the updated vector is projected onto the closed Euclidean ball with radius $r_v$ via
\begin{equation}\label{eq:v_proj_loop}
	v_{i,t+1}^{(m)} = 
	\begin{cases}
		w_{i,t+1}^{(m)}, & \|w_{i,t+1}^{(m)}\| \le r_v, \\
		\dfrac{r_v w_{i,t+1}^{(m)}}{\|w_{i,t+1}^{(m)}\|}, & \text{otherwise}.
	\end{cases}
\end{equation}
After $M$ iterations, we set $v_{i,t+1} = v_{i,t+1}^{(M)}$. The remaining procedures are similar to those in the SDCO-PD algorithm.

The role of these two modifications is to replace the path-dependent recursive tracking errors of the SDCO-PD algorithm  with per-round approximation errors. In particular, the strong convexity of \(g_{i,t}(x,\cdot)\) yields
\begin{align}\label{DDDDD39}
 \|y_{i,t+1}-y_{i,t}^*(x_{i,t})\| \le \frac{\delta}{\mu_g}, 
\end{align}
while the \(M\)-step HvP iteration yields a geometrically decaying error proportional to \(\rho_H^{M/2}\), where $\rho_H = 1 - \lambda \mu_g \in [0, 1)$ for $0 < \lambda < 1 / L_{g,y}$. Consequently, the regret analysis no longer requires recursively accumulating the temporal variations \(H_{1,T}\) and \(E_{1,T}\). This mechanism is formalized in Lemma~\ref{lem:v_bound_dist_multi_norm} and in (\ref{eq:dist_45_norm})--(\ref{eq:Net_Reg_Final_Thm2}).
 Compared with the centralized online bilevel setting in~\cite{Onlinelocalregret}, the adaptive multi-loop updates must here be established jointly with network disagreement and compression perturbations, which enter the distributed primal dynamics.

Compared to the single-loop design of the SDCO-PD algorithm, the AMDCO-PD algorithm based on adaptive multi-loop updates, as summarized in Algorithm~\ref{alg.2}, is proposed to eliminate the explicit dependence of the regret bound on the inner-level and gradient path variations \(H_{1,T}\) and \(E_{1,T}\).

For subsequent analysis, let $K_{i,t} = \min\{k \ge 0 : \Vert{}\nabla_y g_{i,t}(x_{i,t}, y_{i,t+1}^{(k)})\Vert{} \le \delta\}$ denote the number of inner-gradient steps executed by agent $i$ at round $t$, and define $I_{\text{net}} = \sum_{i=1}^{N} \sum_{t=1}^{T} K_{i,t}$.   
\begin{algorithm}[!t] 
	\caption{Adaptive Multi-loop Distributed Compressed Online Primal--Dual (AMDCO-PD) Algorithm} 
	\begin{algorithmic}\label{alg.2}
		\renewcommand{\algorithmicrequire}{\textbf{Input:}}
		\REQUIRE
		Non-increasing $\{ {\alpha _t}\},  \{ {{s _t}} \} \subseteq ( {0, + \infty })$, non-decreasing $\{ {\gamma _t}\} \subseteq ( {0, + \infty })$, step sizes $\beta, \lambda > 0$, multi-loop tolerance $\delta > 0$, HvP iteration limit $M \ge 1$, and radius $r_v\ge C_{f,y}/\mu_g$.
		\renewcommand{\algorithmicrequire}{\textbf{Initialize:}}
		\REQUIRE
		${\hat{z}_{j,0}} = {\mathbf{0}_q}$ for $j \in [ N ]$, and ${z_{i,1}} \in {\mathcal{X}}$, $y_{i,1}= 0_p$, $v_{i,1}= 0_p$, $\varphi_{i,1} = 0_{m_i}$, $\forall i \in [ N ]$.
		\FOR {$t = 1, \cdot  \cdot  \cdot $}
		\FOR {$i = 1,\cdot  \cdot  \cdot, N$ in parallel}
		\STATE Update the projected compressed variable ${\hat{z}_{j,t}}$ via \eqref{Algorithm1-eq1}.
		\STATE Select the local decision variable ${x_{i,t}}$ via \eqref{Algorithm1-eq3}.
		\STATE Observe the current local functions \(f_{i,t}\), \(g_{i,t}\), and \(h_{i,t}\), and access the required first- and second-order derivative oracles at the queried points.
		\STATE Update the local dual variable ${\varphi_{i,t + 1}}$ via \eqref{Algorithm1-eq4}.
		\STATE \textbf{Adaptive Multi-loop for $y$ }
		\STATE Initialize $y_{i,t+1}^{(0)} = y_{i,t}$
		\STATE Set $k = 0$
		\WHILE {$\| \nabla_y g_{i,t}(x_{i,t}, y_{i,t+1}^{(k)}) \| > \delta$}
		\STATE Update $y_{i,t+1}^{(k+1)}$ via \eqref{eq:y_inner_loop}.
		\STATE $k = k + 1$
		\ENDWHILE
		\STATE Set $y_{i,t+1} = y_{i,t+1}^{(k)}$
		
		\STATE \textbf{Multi-loop for HvP estimation $v$}
		\STATE Initialize $v_{i,t+1}^{(0)} = v_{i,t}$
		\FOR {$m = 1$ to $M$}
		\STATE Update $w_{i,t+1}^{(m)}$ via \eqref{eq:w_multi_loop}.
		\STATE Update $v_{i,t+1}^{(m)}$ by projecting $w_{i,t+1}^{(m)}$ onto the $\ell_2$-ball via \eqref{eq:v_proj_loop}.
		\ENDFOR
		\STATE Set $v_{i,t+1} = v_{i,t+1}^{(M)}$
		
		\STATE Approximate the hypergradient $\tilde{\nabla}f_{i,t}(x_{i,t},y_{i,t+1},v_{i,t+1})$ via \eqref{HuGradi}.
			\STATE Update the gradient tracking direction ${{\omega} _{i,t + 1}}$ via \eqref{Algorithm1-eq5}.
		\STATE Update the pre-consensus primal variable $z_{i,t+1}$ via \eqref{Algorithm1-eq6}. 
		\ENDFOR
		\ENDFOR
		\renewcommand{\algorithmicensure}{\textbf{Output:}}
		\ENSURE
		$\{ x_{i,t} \}$.
	\end{algorithmic}
\end{algorithm}

\subsection{Performance Analysis}
In this section, we present the theoretical guarantees of the AMDCO-PD algorithm.

\begin{theorem}\label{thm2}
	Suppose Assumptions~\ref{ass:network_fixed}--\ref{ass.mir} and \ref{ass.cpmpress} hold. For all $i \in [N]$, let $\{x_{i,t}\}$ be the sequences generated by Algorithm~\ref{alg.2} with parameters $\beta \le {2}/({\mu_g + L_{g,y}})$, $\lambda < {1}/{L_{g,y}}$, $\rho_H=1-\lambda \mu_g$, $M =\max  \{1,\lceil -\ln T / \ln \rho_H \rceil\}$, $\delta = \mathcal{O}(T^{-1/2})$ and
	\begin{equation}\label{eq:canshu_alg2}
		\alpha_t = \frac{\alpha_0}{t^{\theta_1}}, \quad \gamma_t = \frac{\gamma_0}{\alpha_t}, \quad s_t=\frac{s_0}{t^{\theta_2}}, \quad \forall t \in \mathbb{N}_+, 
	\end{equation}
	where $\rho_H \in (0,1) $, $\alpha_0, s_0 > 0$, $\theta_1 \in (0, 1)$, $\theta_2=\theta_1+1$, and $\gamma_0 \in (0, \eta/(4G_h^2)]$ are constants. Then, for any $T \in \mathbb{N}_+$, the dynamic network regret and network cumulative constraint violation satisfy
	\begin{align}
		\mathbb{E}[\text{Net-Reg}(T)] &= \mathcal{O}(T^{\max\{\theta_1,1-\theta_1\}} + T^{\theta_1 } P_T),  \label{thm2ret}\\
		\mathbb{E}[\text{Net-CCV}(T)] &= \mathcal{O}(T^{1-\theta_1/2}). \label{thm2Cvv11}
	\end{align}
	The total number of inner iterations $\mathcal{I}_{net}$ satisfies
	\begin{align}\label{totalNum}
	\mathbb{E}[\mathcal{I}_{net}]\le \mathcal{O}(T\log (\sqrt{T})+H_{1,T}).
	\end{align}
	Moreover, if Assumption~\ref{ass.slater} holds, then
	\begin{equation}
		\mathbb{E}[\text{Net-CCV}(T)] = \mathcal{O}(T^{1-\theta_1}).\label{thm2Cvv22} 
	\end{equation}

\end{theorem}
\begin{proof}
	The detailed proof is provided in Appendix~\ref{CCCC}.
\end{proof}

\begin{corollary}\label{cor2}
	Suppose the conditions of Theorem~\ref{thm2} hold. The static network regret and network cumulative constraint violation of Algorithm~\ref{alg.2} satisfy
	\begin{align}
		\mathbb{E}[\text{Net-SReg}(T)] &= \mathcal{O}(T^{\max\{\theta_1,1-\theta_1\}}),\\
		\mathbb{E}[\text{Net-CCV}(T)] &= \mathcal{O}(T^{1-\theta_1/2}).
	\end{align}
	The total number of inner iterations $\mathcal{I}_{net}$ satisfies
	\begin{align}
		\mathbb{E}[\mathcal{I}_{net}]\le \mathcal{O}(T\log (\sqrt{T})+H_{1,T}).
	\end{align}
Furthermore, under Assumption~\ref{ass.slater}, it holds
	\begin{equation}
		\mathbb{E}[\text{Net-CCV}(T)] = \mathcal{O}(T^{1-\theta_1}). 
	\end{equation}
\end{corollary}

\begin{remark}
	Theorems~\ref{thm1} and~\ref{thm2} establish a fundamental trade-off between computational efficiency and robustness against environmental non-stationarity. Algorithm~\ref{algo1} benefits from a single-loop structure, demanding one inner-level update and HvP tracking step per iteration. While this ensures a minimal computational burden, its dynamic network regret inevitably depends on the inner-level path variation, yielding the bound $\mathcal{O}(T^{\max\{\theta_1,1-\theta_1\}}+T^{\theta_1 } P_T+H_{1,T}+E_{1,T})$. In highly volatile environments, drastic environmental changes can cause $H_{1,T}$ and $E_{1,T}$ to grow linearly, thereby precluding a sublinear regret guarantee. This theoretical  limitation necessitates the design of Algorithm~\ref{alg.2}. By implementing an adaptive multi-loop tracking mechanism governed by a gradient-norm tolerance condition ($\|\nabla_y g_{i,t}\|\le\delta$),
	 Algorithm~\ref{alg.2} eliminates the regret's dependence on $H_{1,T}$ and $E_{1,T}$, reducing the bound purely to $\mathcal{O}(T^{\max\{\theta_1,1-\theta_1\}}+T^{\theta_1 } P_T)$.  Nonetheless, this mathematical decoupling is achieved at the expense of an increased computational overhead. To satisfy the tolerance condition, the adaptive mechanism dynamically invokes multiple inner iterations per-round, which accumulates a larger number of gradient and HvP evaluations over the horizon. Consequently, while Algorithm~\ref{alg.2} demands a larger computational budget, it achieves better performance under severe environmental variations. 
\end{remark}  %%%%%%%%%%%%%%%%%%%%%

\section{Numerical Example}\label{sec5}
 We consider a multi-agent network comprising $N=50$ agents. The communication topology is characterized by a doubly stochastic mixing matrix $W$, configured with $[W]_{ii}=0.5$ for $i\in [N]$ and $[W]_{i,i+1}=0.5$ for $i=1,\ldots,N-1$, and $[W]_{N,1}=0.5$, with all other entries being zero. The network cooperatively solves a distributed online  bilevel optimization problem formulated as
	\begin{align}
		&   \mathop {\min }\limits_{x \in \mathcal{X}} \;\;\;\textstyle  \frac{1}{N}  \textstyle  \sum_{i = 1}^N {f_{i,t}(x, y_{i,t}^*(x))} , \nonumber\\
		&\textstyle  \;{\rm{s}}{\rm{.t}}{\rm{.  }}\;\;\;\; y_{i,t}^*(x) = \mathop {\arg \min }\limits_{y \in \mathbb{R}^p} g_{i,t}(x,y), \label{example2-2}\\
		&\;\;\;\;\;\;\;\;\;\;\; \textstyle  {B_{i,t}}x - {b_{i,t}} \le {\mathbf{0}_{{m_i}}},\quad \forall i \in [ N ],\quad \forall t \in [ T ], \nonumber
	\end{align} 
where $\mathcal{X} = [-5,5]^q$  with $q=2$, $p=2$  and  $m_i=2$. The outer-level cost function is defined as
$$f_{t}(x,y) = \begin{cases} \frac{1}{2} \Vert{}x - y - c_{t}\Vert{}^2, & \text{if } \Vert{}x - y - c_{t}\Vert{} \le R, \\ R \Vert{}x - y - c_{t}\Vert{} - \frac{1}{2} R^2, & \text{if } \Vert{}x - y - c_{t}\Vert{} > R, \end{cases}$$
where $R=60$.
The inner-level strongly convex cost function is given by \begin{align*}
	{g_{i,t}}(x,y) = \frac{1}{2}(\mu  + \pi  ){\| y \|^2} - \mu {y^T}{H_{i,t}}x, 
\end{align*}
where $\mu = 1$ and $\pi = 0.5$. The bilevel instance can be considered as a stylized hierarchical resource allocation problem over a communication network. Specifically, \(x\) represents the common upper-level control variable, \(y_{i,t}^{\ast}(x)\) represents the local resource allocation induced by this control, and the time-varying parameters \(H_{i,t}\), \(c_{i,t}\), \(B_{i,t}\), and \(b_{i,t}\) capture temporal changes in local operating conditions and capacity limits. In the simulations, we use the Euclidean mirror map $\psi(x) = \frac{1}{2}\Vert{}x\Vert{}^2$.

We design two distinct dynamic environments by manipulating time-varying parameters in  cost functions to artificially govern the inner-level path variation and gradient path variation. 

In a moderate environment, the time-varying  parameters are dynamically generated with a decay factor $t^{-0.5}$, 
\begin{equation}\label{AA51}
\begin{aligned}
	H_{i,t} &= 1.5 I + \big(0.2 + 0.5 t^{-0.5} \sin(0.02 t)\big) \mathbf{1}_{p \times q}, \\
	c_{i,t} &= \begin{bmatrix}
		5 + 10t^{-0.5} \sin(0.015 t) \\
		5+ 10t^{-0.5} \cos(0.015 t)
	\end{bmatrix}.
\end{aligned}
\end{equation}
  To verify the tightened network cumulative constraint violation bounds, we evaluate two specific configurations. For the configuration where Slater's condition holds, we set $b_{i,t}=0.01+\tau_{i,t}$ with $\tau_{i,t}$ uniformly drawn from $[0, 1]$, and randomly generate $B_{i,t}\in[0,1]^{m_i\times q}$. Conversely, for the configuration without Slater's condition, we set $b_{i,t}=0$, randomly generate a nonzero $q$-dimensional vector $a_{i,t}$ whose entries are independently drawn from $[0, 1]$, and set $B_{i,t} = [a_{i,t}, -a_{i,t}]^T$. In this case, set the dynamic step sizes as $\alpha_t = 0.24/\sqrt{t}$, $\gamma_t= 0.4t^{0.5}$, and $s_t = 1.0/t^{1.5}$, and set the fixed parameters as $\beta = 0.4$, $\lambda = 0.5$, and $r_v=40$.

For the algorithmic implementation, both algorithms employ a uniform quantizer $\mathcal{C}(x)= \Delta \text{round}(x / \Delta)$ with $\Delta=1.0$.  For Algorithm~\ref{alg.2}, the adaptive multi-loop tracking mechanism operates with $\alpha_t = 0.06/\sqrt{t}$,  $\beta=0.4$ and $\delta=10^{-3}$. Furthermore, its multi-step HvP estimation executes for a fixed $M=10$ steps per-round with $\lambda=0.5$ and $r_v=40$. For comparison, we also implement the variants FSOBO-PD and AOBO-PD of the centralized online bilevel algorithms in~\cite{Onlinelocalregret}. Since these algorithms are originally developed without time-varying inequality constraints, we augment them with the same dual update in (\ref{Algorithm1-eq4}) so that all algorithms are evaluated under the constrained problem in (\ref{example2-2}).

\begin{figure}[htbp]
	\centering
	\includegraphics[width=0.7\linewidth]{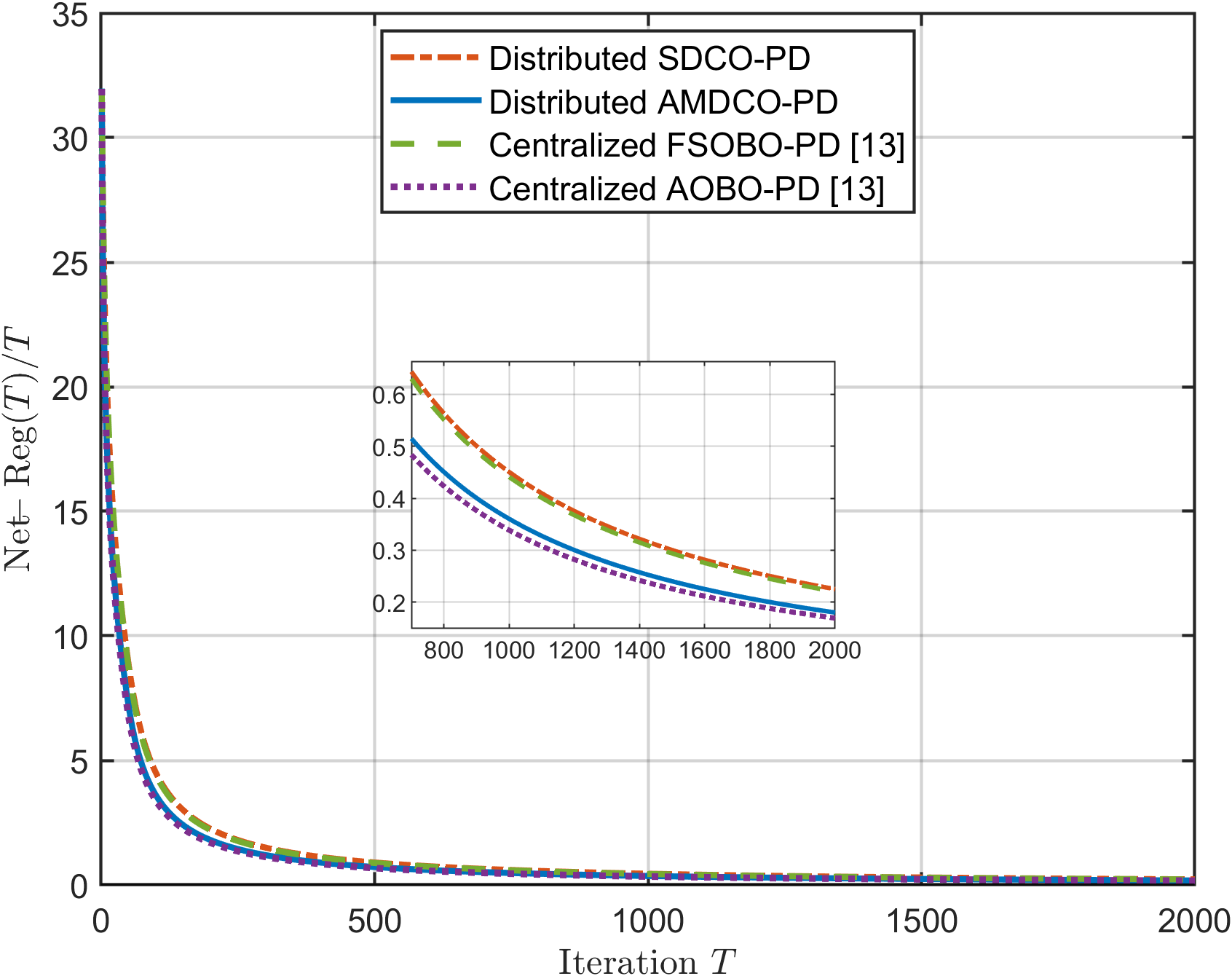}
	\caption{Evolution of the time-averaged dynamic network regret under the moderate environment.}
	\label{Fig11}
\end{figure}

\begin{figure}[htbp]
	\centering
	\includegraphics[width=0.7\linewidth]{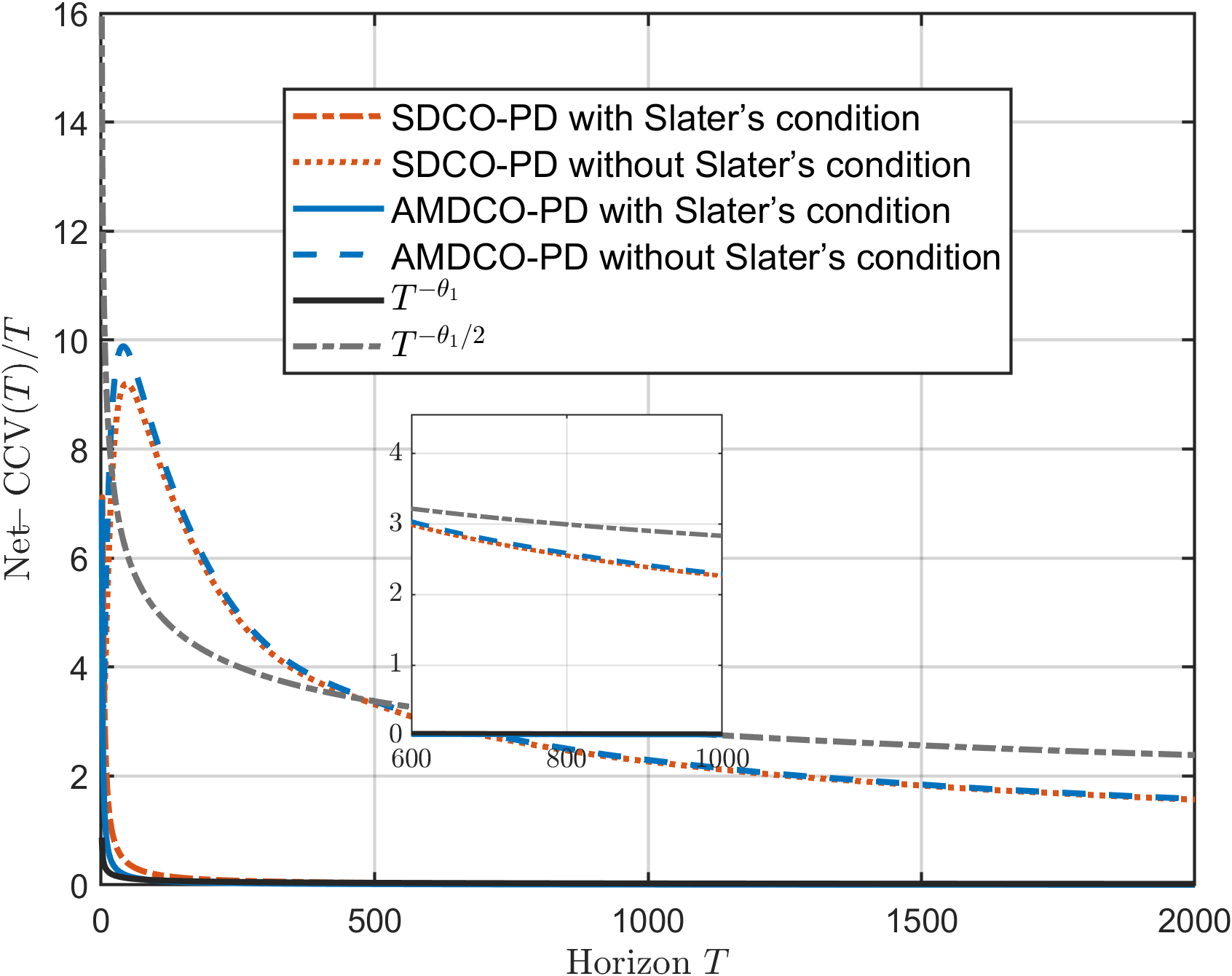}
	\caption{Evolution of the time-averaged network cumulative constraint violation  under the moderate environment.}
	\label{Fig22}
\end{figure}

\begin{figure}[htbp]
	\centering
	\includegraphics[width=0.7\linewidth]{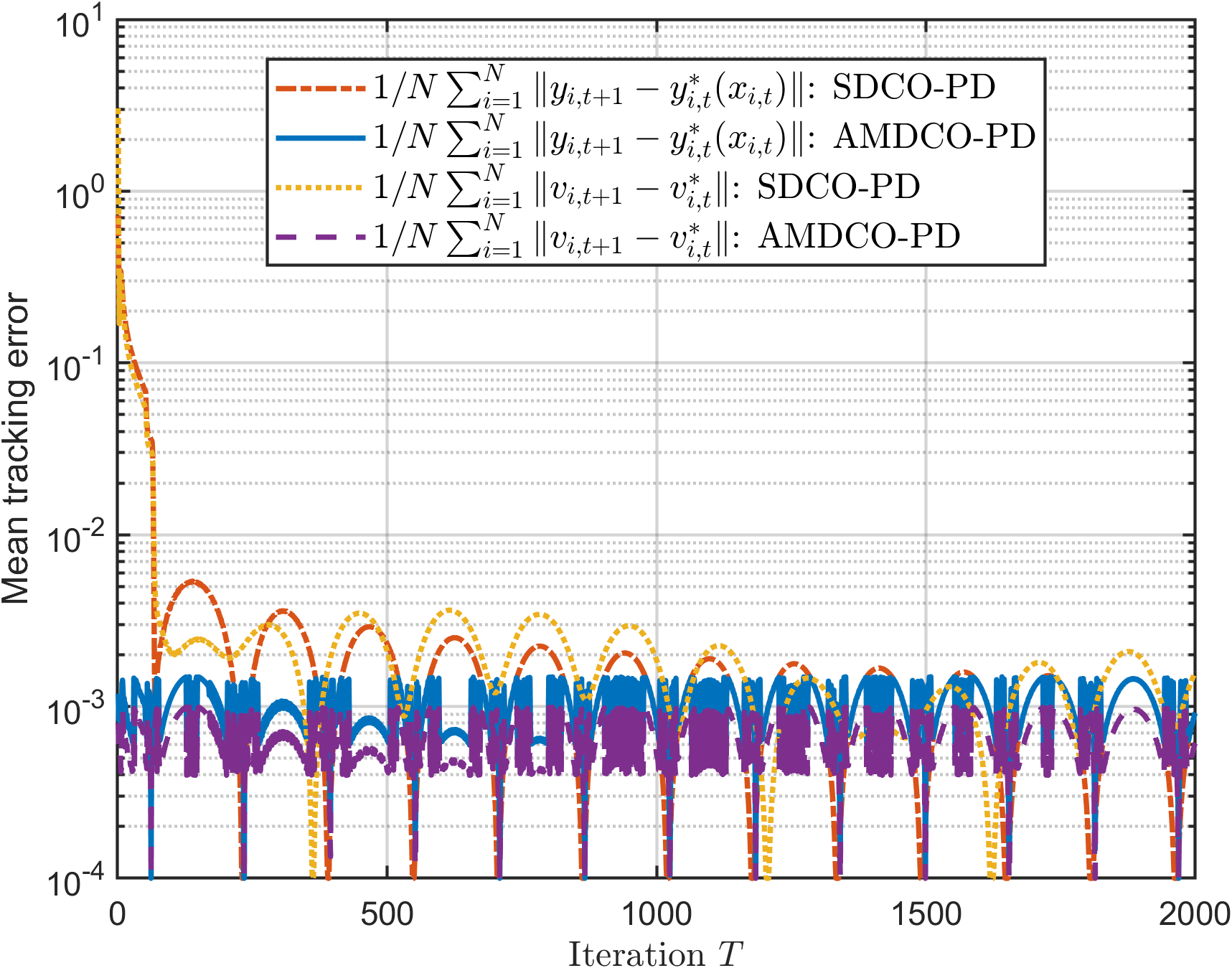}
	\caption{Evolution of the mean inner-level tracking errors $1/N\sum_{i=1}^{N}\|y_{i,t+1} - y_{i,t}^*(x_{i,t})\|$ and $1/N\sum_{i=1}^{N}\|v_{i,t+1} - v_{i,t}^*\|$ under the moderate environment.}
	\label{Fig33}
\end{figure}

\begin{figure}[htbp]
	\centering
	\includegraphics[width=0.7\linewidth]{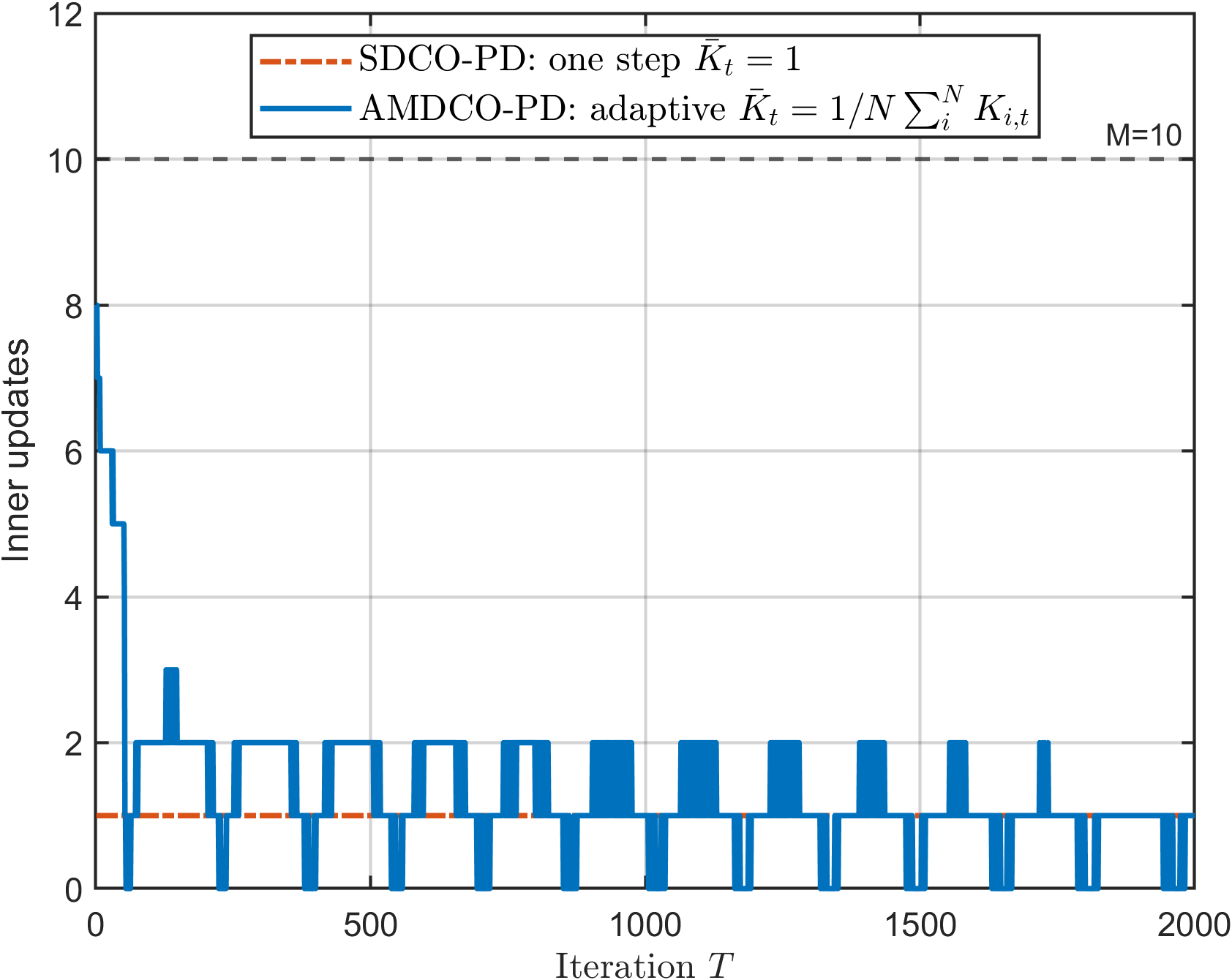}
	\caption{Average number of inner-gradient updates  under the extreme environment.}
	\label{Fig44}
\end{figure}

The numerical results in the moderate environment are illustrated in Fig.~\ref{Fig11} and Fig.~\ref{Fig22}. In particular, Fig.~\ref{Fig11} depicts the evolution of the time-averaged network dynamic regret \(\mathrm{Net\mbox{-}Reg}(T)/T\), together with the centralized baselines algorithms adapted from~\cite{Onlinelocalregret}. It is observed that all algorithms drive the time-averaged network dynamic regret toward zero. Although the centralized baselines achieve slightly lower regret, the proposed distributed algorithms exhibit comparable decay trends under compressed communication. Moreover, AMDCO-PD attains a slightly lower regret trajectory than SDCO-PD, which is consistent with the improved inner-level tracking accuracy of its adaptive mechanism. 
Fig.~\ref{Fig22} depicts the evolution of the time-averaged network cumulative constraint violation \(\mathrm{Net\mbox{-}CCV}(T)/T\).  
Under Slater’s condition, the time-averaged network cumulative constraint violations decay consistently with the \(T^{-\theta_1}\) reference rate. Without Slater’s condition, the trajectories exhibit the slower \(T^{-\theta_1/2}\) decay predicted by Theorems~\ref{thm1} and~\ref{thm2}.
This confirms that the time-averaged network cumulative constraint violation of SDCO-PD algorithm and AMDCO-PD algorithm can be improved under Slater's condition.

To explicitly highlight the theoretical vulnerabilities of single-loop architectures, we further construct an extreme environment aimed at severely amplifying the inner-level variations. This is achieved by applying a slower decay factor of $t^{-0.25}$ alongside high-frequency and large-amplitude perturbations, generated as
\begin{equation}\label{AA52}
	\begin{aligned}
		H_{i,t} &= 1.5 I + \big(0.2 + 5.0 t^{-0.25} \sin(0.3 t)\big) \mathbf{1}_{p \times q}, \\
		c_{i,t} &= \begin{bmatrix}
			5 + 40 t^{-0.25} \sin(0.4 t) \\
			5 + 40 t^{-0.25} \cos(0.4 t)
		\end{bmatrix}.
	\end{aligned}
\end{equation}
Compared to the moderate setup in (\ref{AA51}), the parameters in (\ref{AA52}) feature a significantly slower decay rate and much higher-frequency oscillations. This configuration intentionally simulates a highly volatile and severe dynamic environment. 

To elucidate the structural distinction between the single-loop and multi-loop mechanisms, we analyze the mean tracking errors $1/N\sum_{i=1}^{N}\|y_{i,t+1}-y_{i,t}^*(x_{i,t})\|$ and $1/N\sum_{i=1}^{N}\|v_{i,t+1}-v_{i,t}^*\|$ alongside the inner-loop updates per iteration across both environments. Fig.~\ref{Fig33} and Fig.~\ref{Fig44} depict the moderate environment, while Fig.~\ref{Fig77} and Fig.~\ref{Fig88} characterize the extreme environment. In Fig.~\ref{Fig33}, while the SDCO-PD algorithm maintains bounded tracking errors, its precision is bottlenecked by the single-step update mechanism. Conversely, the AMDCO-PD algorithm achieves lower tracking errors. As depicted in Fig.~\ref{Fig44}, compared to the fixed single-step inner update of the SDCO-PD algorithm, the AMDCO-PD algorithm dynamically invokes multiple inner iterations to enforce the gradient-norm tolerance condition $\delta$, alongside a fixed  $M=10$  steps for HvP estimation.  It is worth noting that under the moderate environment, the performance of both algorithms is relatively close. As depicted in Fig.~\ref{Fig33} and Fig.~\ref{Fig44}, there is no significant difference in their average tracking errors or cumulative inner updates.

In Fig.~\ref{Fig77}, the artificially inflated environmental variations cause severe tracking degradation in the SDCO-PD algorithm, manifesting as high-amplitude oscillations in the mean tracking error.  In contrast, the AMDCO-PD algorithm maintains a similar tracking accuracy to that in the moderate environment of Fig.~\ref{Fig33}, even under more severe dynamic variations. As shown in Fig.~\ref{Fig88}, the adaptive mechanism automatically scales the computational budget, frequently executing 5 to 10 inner updates per iteration. This dynamic allocation yields the stable and high-precision tracking trajectory observed in Fig.~\ref{Fig77}. These results illustrate that the adaptive mechanism maintains small per-round inner-level and HvP tracking errors under larger temporal variations, which is consistent with the error-control mechanism underlying Theorem~\ref{thm2}.

%These results numerically substantiate Theorem~\ref{thm2}, confirming that the AMDCO-PD algorithm successfully decouples the dynamic network regret from inner-level path variation at the expense of an adaptively allocated computational budget.

\begin{figure}[htbp]
	\centering
	\includegraphics[width=0.7\linewidth]{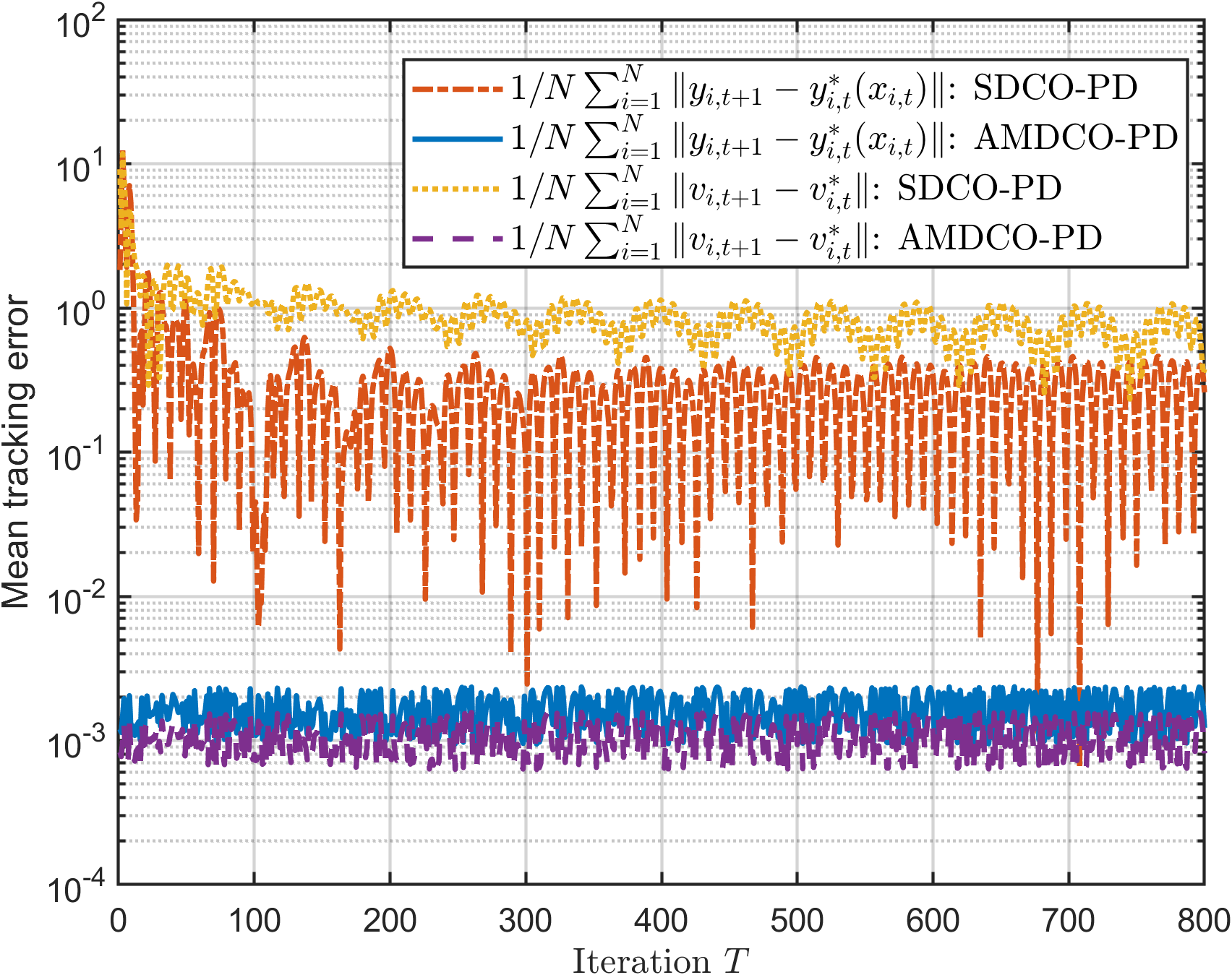}
	\caption{Evolution  of the mean inner-level tracking errors $1/N\sum_{i=1}^{N}\|y_{i,t+1} - y_{i,t}^*(x_{i,t})\|$ and $1/N\sum_{i=1}^{N}\|v_{i,t+1} - v_{i,t}^*\|$  under the extreme environment. }
	\label{Fig77}
\end{figure}

\begin{figure}[htbp]
	\centering
	\includegraphics[width=0.7\linewidth]{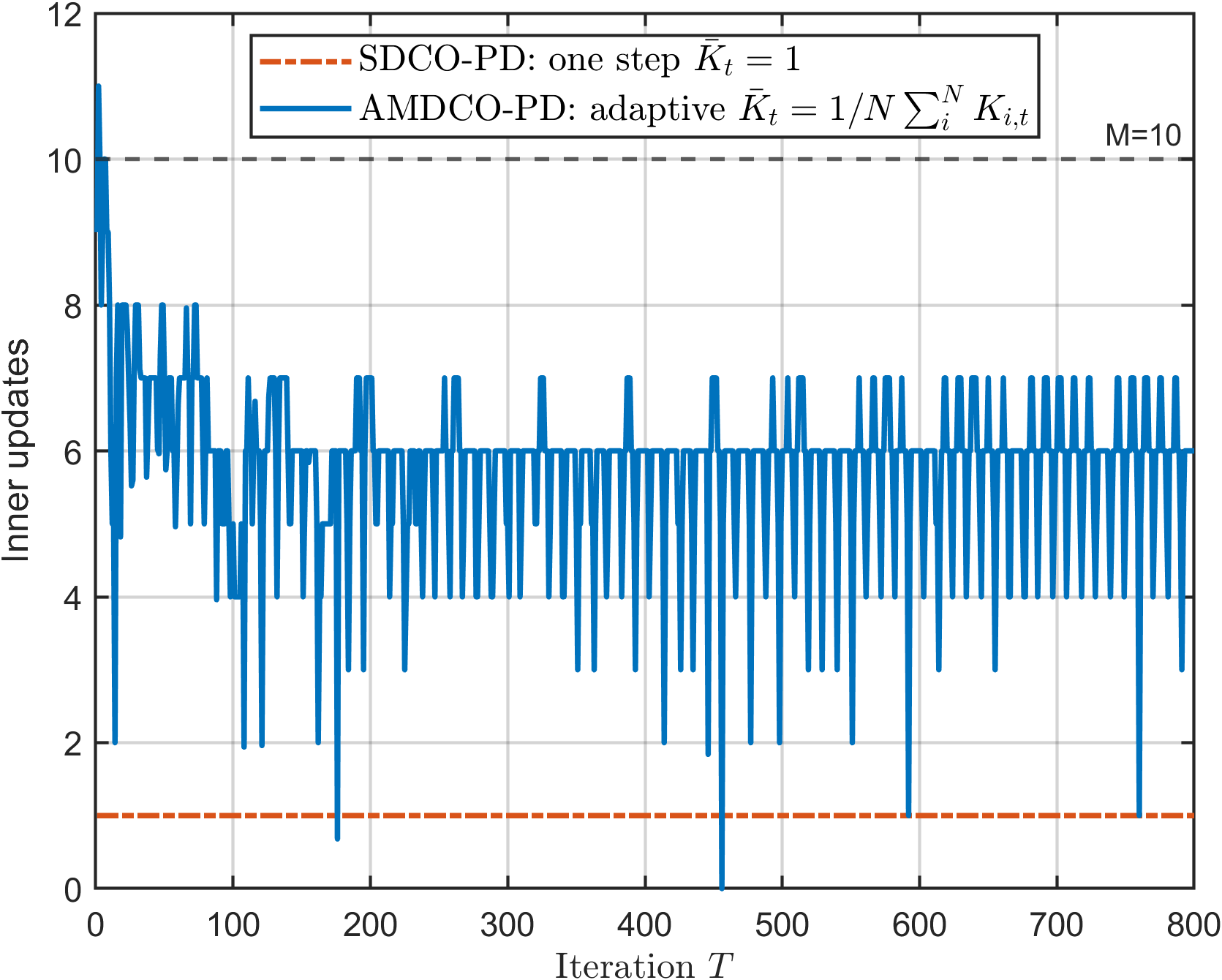}
	\caption{Number of inner updates per iteration under the extreme environment.}
	\label{Fig88}
\end{figure}

\section{Conclusions}\label{sec6}
This paper investigated the distributed online  bilevel optimization problem subject to time-varying inequality constraints. We proposed two communication-efficient primal--dual algorithms. First, a single-loop distributed algorithm was developed by integrating an iterative approximation approach with communication compression. We analyzed its dynamic network regret alongside the network cumulative constraint violation that is tightened under Slater’s condition. Furthermore, to eliminate the regret's dependence on the inner-level and gradient path variation, we introduced a novel adaptive multi-loop algorithm. By dynamically adjusting the number of inner-level updates based on a gradient-norm tolerance condition, this algorithm achieved an improved dynamic network regret bound that relies solely on the outer-level path variation. Future research will explore extending these frameworks to fully asynchronous communication protocols and online bilevel optimization under zeroth-order feedback.

\appendix
\subsection{Useful Lemmas}
We first present auxiliary lemmas. These preliminary properties are subsequently synthesized to derive the dynamic network regret and network cumulative constraint violation bounds.

\begin{lemma}[{\cite[Lemma 1]{yang2023achieving}}]\label{lemVM}
	If Assumptions~\ref{ass:outer} and \ref{ass:inner} hold, then for any $i\in [N]$, $t\in [T]$, $x \in \mathcal{X}$ and $y \in \mathbb{R}^p$, it holds that
	\begin{equation}\label{equ:VM}
	\textstyle	\|v_{i,t}(x, y)\|^2 \le \frac{C_{f,y}^2}{\mu_g^2}.
	\end{equation}
\end{lemma}
\begin{lemma}[{\cite[Proposition 1]{niu2025lopa}}]\label{lem.11}
	Suppose Assumptions~\ref{ass:outer} and \ref{ass:inner} hold. Let $\bar{\nabla}f_{i,t}(x,y) = \nabla_x f_{i,t}(x,y) - \nabla_{xy}^2 g_{i,t}(x,y)v_{i,t}(x,y)$ be a surrogate of the local hypergradient $\nabla f_{i,t}(x,y_{i,t}^*(x))$. Then given any $x, x' \in \mathcal{X}$ and $y, y' \in \mathbb{R}^p$, it holds that for all $i \in [N]$ and $t \in [T]$:
	\begin{align}
		&\textstyle\| y_{i,t}^*(x) - y_{i,t}^*(x') \| \le L_y \| x - x' \|, \label{eq:lip_y_star} \\
		&\textstyle \| v_{i,t}(x, y) - v_{i,t}(x', y') \| \le L_v \big(\| x - x' \| + \| y - y' \|\big), \label{eq:lip_v} \\
		&\textstyle \| v_{i,t}^*(x) - v_{i,t}^*(x') \| \le L_{v^*} \| x - x' \|, \label{eq:lip_v_star} \\
		&\textstyle \| \bar{\nabla}f_{i,t}(x, y) - \bar{\nabla}f_{i,t}(x', y') \| \le L_{\bar{f}} \big(\| x - x' \| + \| y - y' \|\big), \label{eq:lip_surrogate} \\
		&\textstyle\| \nabla f_t(x) - \nabla f_t(x') \| \le L \| x - x' \|, \label{eq:lip_hypergradient}
	\end{align}
where the Lipschitz constants are defined as follows: $L_y = C_{g,xy} / \mu_g$, $L_v = L_{f,y} / \mu_g + C_{f,y} L_{g,yy} / \mu_g^2$, $L_{v^*} = L_v(1 + L_y)$, $L_{\bar{f}} = L_{f,x} + C_{g,xy} L_v + C_{f,y} L_{g,xy} / \mu_g$, and $L = L_{\bar{f}}(1 + L_y)$.
\end{lemma}

\begin{lemma}\label{lem.22}
	Suppose Assumption~\ref{ass:network_fixed} holds. For all $i, j \in [N]$ and $t \in \mathbb{N}_+$, the sequences $\{x_{i,t}\}$ generated by Algorithm~\ref{algo1} or Algorithm~\ref{alg.2} satisfy
	\begin{align}
	&\textstyle 	\frac{1}{N} \sum_{t=1}^{T} \sum_{i=1}^{N} \sum_{j=1}^{N} \|x_{i,t} - x_{j,t}\|\nonumber\\
	 & \quad \quad \textstyle \leq N\vartheta_1 + \vartheta_2 \sum_{t=1}^{T} \sum_{i=1}^{N} \|\varepsilon_{i,t}^z\|, \label{Cons11}\\
		&\textstyle \frac{1}{N} \sum_{t=1}^{T} \sum_{i=1}^{N} \sum_{j=1}^{N} \|x_{i,t} - x_{j,t}\|^2 \nonumber\\
		&\qquad \textstyle \leq \vartheta_3 + \textstyle \vartheta_4 \sum_{t=1}^{T} \sum_{i=1}^{N} \|\varepsilon_{i,t}^z\|^2,\label{Cons22}
	\end{align}
	where $\tau = (1 - \underline{w}/(4N^2))^{-2} > 1$, $\sigma = 1 - \underline{w}/(4N^2) \in (0,1)$, $\varepsilon_{i,t}^z = \hat{z}_{i,t+1} - x_{i,t}$, $\vartheta_1 = 2\tau/(\sigma(1-\sigma)) \sum_{j=1}^{N} \|\hat{z}_{j,1}\|$, $\vartheta_2 = 4 + 2N\tau/(1-\sigma)$, $\vartheta_3 = 16N\tau^2/(\sigma^2(1-\sigma^2)) \big(\sum_{i=1}^{N} \|\hat{z}_{i,1}\|\big)^2$, and $\vartheta_4 = 16N^2\tau^2/(1-\sigma)^2 + 32$.
\end{lemma}
\begin{proof}
The proof follows from the similar derivation as in~\cite{Xinlei2025}.
\end{proof}

\begin{lemma}\label{lem.55}
	Suppose Assumptions~\ref{ass:network_fixed}--\ref{ass.mir} and \ref{ass.cpmpress} hold. For all $i \in [N]$, let $\{x_{i,t}\}$ be the sequences generated by Algorithm~\ref{algo1} or Algorithm~\ref{alg.2}  with $\gamma_t = \gamma_0 / \alpha_t$, where $\gamma_0 \in (0, \eta/ (4G_h^2)]$ is a constant. Then, for any $T \in \mathbb{N}_+$ and $\hat{x}_{[T+1]} = \{\hat{x}_t\}_{t=1}^{T+1}$ satisfying $\hat{x}_t \in \mathcal{X}$, $h_t(\hat{x}_t)\le 0_m$, it holds that
	\begin{align}\label{Lemma5-eq1}
		&\textstyle \frac{1}{N}\sum_{i = 1}^N \sum_{t = 1}^T \mathbb{E}\left[ \langle \tilde{\nabla}f_{i,t}(x_{i,t},y_{i,t+1},v_{i,t+1}), x_{i,t} - \hat{x}_t \rangle\right]   \nonumber\\
		&\le \textstyle\frac{2G_s^2}{\eta}\sum_{t = 1}^T \alpha_t + \frac{1}{N}\sum_{i = 1}^N \sum_{t = 1}^T \mathbb{E}[\Delta_{i,t}(\hat{x}_t)]   \nonumber\\
		&\quad \textstyle-  \frac{1}{N}\sum_{i = 1}^N \sum_{t = 1}^T \mathbb{E}\left[\frac{\eta\|x_{i,t} - z_{i,t + 1}\|^2}{4\alpha_t}\right]  \nonumber\\
		&\quad \textstyle+ \sum_{t = 1}^T \frac{1}{\alpha_t} \left( d(\mathcal{X}) L_{\psi} \sqrt{C} s_{t + 1} + L_{\psi}  C s_{t + 1}^2 \right),  \\
		&\textstyle \sum_{i = 1}^N \sum_{t = 1}^T \frac{1}{2} \mathbb{E}\left[ \frac{\varphi_{i,t + 1}^T h_{i,t}(x_{i,t})}{\gamma_t} + \frac{\eta \|x_{i,t} - z_{i,t + 1}\|^2}{2\gamma_0} \right] \nonumber\\
		&\textstyle \le \sum_{i = 1}^N \sum_{t = 1}^T \mathbb{E}\left[ \frac{\varphi_{i,t + 1}^T h_{i,t}(\hat{x}_t)}{\gamma_t} \right] + \tilde{\Lambda}_T(\hat{x}_{[T+1]}),  \label{Lemma5-eq2}
	\end{align}
	where $G_s = C_{f,x} + C_{g,xy} r_v$, $\Delta_{i,t}(\hat{x}_t)=( \mathcal{D}_{\psi}(\hat{x}_t, x_{i,t}) -\mathcal{D}_{\psi}(\hat{x}_{t}, x_{i,t+1}))/{\alpha_t}$  and 
	\begin{align}\label{Lambda_def}
		&\textstyle \tilde{\Lambda}_T(\hat{x}_{[T+1]}) = \frac{1}{\gamma_0}\sum_{i = 1}^N \mathbb{E}[{\mathcal{D}_\psi(\hat{x}_1, x_{i,1})}] \nonumber\\
		& \textstyle+ \frac{1}{\gamma_0}\sum_{i=1}^N\sum_{t=1}^{T} \mathbb{E}\left[ \left| \mathcal{D}_\psi(\hat{x}_{t+1}, x_{i,t+1}) - \mathcal{D}_\psi(\hat{x}_t, x_{i,t+1}) \right| \right] \nonumber\\
		& \textstyle+  2NG_s d(\mathcal{X})\sum_{t = 1}^T \frac{1}{\gamma_t} + \frac{2N\gamma_0 G_s^2}{\eta}\sum_{t = 1}^T \frac{1}{\gamma_t^2} \nonumber\\
		& \textstyle+ \frac{N}{\gamma_0}\sum_{t = 1}^T \left( d(\mathcal{X})L_{\psi}\sqrt{C}s_{t} + L_{\psi}  C s_{t}^2 \right).
	\end{align}
\end{lemma}

\ifbool{submitA}{
	\begin{proof}
		The proof is given in Appendix A of the arXiv version \cite{Mypaper} and is omitted here due to space limitation.
	\end{proof} 
}{
\begin{proof}
(i) By the convexity of $h_{i,t}$ and the Cauchy--Schwarz inequality, for any $\hat{x}_t \in \mathcal{X}$, we expand the inner product as follows
\begin{align}
	&\langle \tilde{\nabla}f_{i,t}(x_{i,t},y_{i,t+1},v_{i,t+1}), x_{i,t} - \hat{x}_t \rangle \nonumber\\
	&= \langle \tilde{\nabla}f_{i,t}(x_{i,t},y_{i,t+1},v_{i,t+1}), x_{i,t} - z_{i,t+1} \rangle \nonumber\\
	&\quad + \langle \omega_{i,t+1}, z_{i,t+1} - \hat{x}_t \rangle + \langle \nabla h_{i,t}(x_{i,t})^T \varphi_{i,t+1}, \hat{x}_t - z_{i,t+1} \rangle \nonumber\\
	&\le G_s \|x_{i,t} - z_{i,t+1}\| + \varphi_{i,t+1}^T \big(h_{i,t}(\hat{x}_t) - h_{i,t}(x_{i,t})\big) \nonumber\\
	&\quad + G_h \|\varphi_{i,t+1}\| \|x_{i,t} - z_{i,t+1}\|  + \langle \omega_{i,t+1}, z_{i,t+1} - \hat{x}_t \rangle, \label{eq:inner_product_bound}
\end{align}
where the inequality holds due to  \eqref{Algorithm1-eq4}, the definition of $\omega_{i,t+1}$ in \eqref{Algorithm1-eq5}, $\varphi_{i,t+1} \ge \mathbf{0}_{m_i}$ and the fact that
\begin{align}\label{Ms}
&\|\tilde{\nabla}f_{i,t}(x_{i,t},y_{i,t+1},v_{i,t+1})\|\nonumber\\
&\le  \|\nabla_x f_{i,t}(x_{i,t}, y_{i,t+1})\| + \|\nabla_{xy}^2 g_{i,t}(x_{i,t}, y_{i,t+1}) v_{i,t+1}\|\nonumber\\
&\le C_{f,x} + C_{g,xy} r_v= G_s.
\end{align}

Applying \cite[Lemma 3]{Xinlei2025} to \eqref{Algorithm1-eq6} and utilizing Assumptions~\ref{ass:network_fixed} and \ref{ass.mir}, the last term in \eqref{eq:inner_product_bound} is bounded as
	\begin{align}
		&\langle \omega_{i,t+1}, z_{i,t+1} - \hat{x}_t \rangle \nonumber\\
		&\le \frac{1}{\alpha_t} \big( \mathcal{D}_{\psi}(\hat{x}_t, x_{i,t}) - \mathcal{D}_{\psi}(\hat{x}_t, z_{i,t+1}) - \mathcal{D}_{\psi}(z_{i,t+1}, x_{i,t}) \big) \nonumber\\
		&\le \Delta_{i,t}(\hat{x}_t) + \frac{1}{\alpha_t} \bigg( \sum_{j=1}^N [W]_{ij} \mathcal{D}_{\psi}(\hat{x}_t, \hat{z}_{j,t+1}) \nonumber\\
		&\quad - \mathcal{D}_{\psi}(\hat{x}_t, z_{i,t+1}) - \frac{\eta}{2}\| z_{i,t+1} - x_{i,t} \|^2 \bigg), \label{eq:omega_bound}
	\end{align}
where the last inequality is due to \eqref{strong} and \eqref{Algorithm1-eq3}.
	
By (\ref{Algorithm1-eq1}), we obtain $\hat{z}_{i,t+1} \in \mathcal{X}$. Since $z_{i,t+1}$, $\hat{x} \in \mathcal{X}$, applying (\ref{strong}), (\ref{three}) and (\ref{KKbound}), we have
\begin{align}
	&\mathcal{D}_\psi(\hat{x}, \hat{z}_{i,t+1}) - \mathcal{D}_\psi(\hat{x}, z_{i,t+1}) + \mathcal{D}_\psi(\hat{z}_{i,t+1}, z_{i,t+1}) \nonumber\\
	&= \langle \hat{x} - \hat{z}_{i,t+1}, \nabla\psi(z_{i,t+1}) - \nabla\psi(\hat{z}_{i,t+1}) \rangle \nonumber\\
	&\le d(\mathcal{X}) L_{\psi} \|z_{i,t+1} - \hat{z}_{i,t+1}\| + L_{\psi} \|z_{i,t+1} - \hat{z}_{i,t+1}\|^2. \label{eq:bregman_diff_bound}
\end{align}

Define $\tilde{z}_{i,t+1} = \hat{z}_{i,t} + s_{t+1} \mathcal{C} \left(  ({z_{i,t+1} - \hat{z}_{i,t}})/{s_{t+1}} \right)$. Since $z_{i,t+1} \in \mathcal{X}$, we have $\mathcal{P}_{\mathcal{X}}(z_{i,t+1}) = z_{i,t+1}$. By the non-expansiveness of the projection and Assumption~\ref{ass.cpmpress}, we have
\begin{align}\label{eq:compression_bound}
	&\mathbb{E} \big[ \|z_{i,t + 1} - \hat{z}_{i,t + 1}\|^2 \big] \nonumber\\
	&= \mathbb{E}\left[ \| \mathcal{P}_{\mathcal{X}}(z_{i,t+1}) - \mathcal{P}_{\mathcal{X}}(\tilde{z}_{i,t+1}) \|^2 \right]\nonumber\\ &\le \mathbb{E}\left[ \| z_{i,t+1} - \tilde{z}_{i,t+1} \|^2 \right]\nonumber\\
	&\le s_{t + 1}^2 \mathbb{E} \bigg[ \bigg\| \frac{z_{i,t + 1} - \hat{z}_{i,t}}{s_{t + 1}} - \mathcal{C}\bigg( \frac{z_{i,t + 1} - \hat{z}_{i,t}}{s_{t + 1}} \bigg) \bigg\|^2 \bigg] \nonumber\\
	&\le  C s_{t + 1}^2.
\end{align}

	Applying Young's inequality, we have
	\begin{align}
		&\textstyle \big( G_s + G_h\|\varphi_{i,t + 1}\| \big) \|x_{i,t} - z_{i,t + 1}\| \nonumber\\
		&\textstyle \le \frac{2G_s^2\alpha_t}{\eta} + \frac{2G_h^2\alpha_t}{\eta}\|\varphi_{i,t + 1}\|^2 + \frac{\eta \|x_{i,t} - z_{i,t + 1}\|^2}{4\alpha_t}. \label{eq:youngs_ineq}
	\end{align}

	From \eqref{Algorithm1-eq4}, we have $\|\varphi_{i,t+1}\| = \gamma_t \|[h_{i,t}(x_{i,t})]_+\|$. Substituting $\gamma_t = \gamma_0 / \alpha_t$ and using the identity $[h_{i,t}(x_{i,t})]_{+}^T h_{i,t}(x_{i,t}) = \Vert{}[h_{i,t}(x_{i,t})]_{+}\Vert{}^2$, we have
	\begin{align}
		&\textstyle \frac{2G_h^2\alpha_t}{\eta}\|\varphi_{i,t+1}\|^2 - \varphi_{i,t+1}^T h_{i,t}(x_{i,t}) \nonumber\\
		&= \textstyle \left( \frac{2G_h^2\alpha_t \gamma_t^2}{\eta} - \gamma_t \right) \|[h_{i,t}(x_{i,t})]_+\|^2 \nonumber\\
		&= \textstyle \gamma_t \left( \frac{2G_h^2 \gamma_0}{\eta} - 1 \right) \|[h_{i,t}(x_{i,t})]_+\|^2 \le 0, \label{eq:penalty_cancel}
	\end{align}
	where the inequality holds by setting $\gamma_0 \le \eta / (2G_h^2)$. 
	
	Combining \eqref{eq:inner_product_bound}--\eqref{eq:penalty_cancel} and leveraging $h_{i,t}(\hat{x}_t) \le \mathbf{0}_{m_i}$, we sum the  inequalities over $i \in [N]$. Subsequently, taking the expectation, using the doubly stochastic property $\sum_{i=1}^N [W]_{ij} = 1$, and  summing over $t \in [T]$, we arrive at \eqref{Lemma5-eq1}.

	(ii)  Dividing both sides by $\gamma_t$ of (\ref{eq:inner_product_bound}), we have
	\begin{align}\label{Deee}
		&\textstyle \frac{\varphi_{i,t+1}^T h_{i,t}(x_{i,t})}{\gamma_t} \leq  \frac{\langle \tilde{\nabla} f_{i,t}(x_{i,t},y_{i,t+1},v_{i,t+1}), \hat{x}_{t}-x_{i,t} \rangle}{\gamma_t} \nonumber\\
		&\qquad\textstyle + \frac{1}{\gamma_t} (G_s + G_h \|\varphi_{i,t+1}\|) \|x_{i,t} - z_{i,t+1}\| \nonumber\\
		&\qquad\textstyle + \frac{\varphi_{i,t+1}^T h_{i,t}(\hat{x}_t)}{\gamma_t} + \frac{\langle \omega_{i,t+1}, z_{i,t+1} - \hat{x}_t \rangle}{\gamma_t}. 
	\end{align}
	
	By (\ref{Ms}) and the definition of \(d(\mathcal X)\), we have
\begin{align}
	\langle \tilde{\nabla}f_{i,t}(x_{i,t},y_{i,t+1},v_{i,t+1}), \hat{x} - x_{i,t} \rangle \le 2 d(\mathcal{X}) G_s.
\end{align}
	
 Substituting $\alpha_t \gamma_t = \gamma_0$, it follows from (\ref{eq:youngs_ineq}) that
	\begin{align}
		&\textstyle \frac{1}{\gamma_t}\big( G_s + G_h\|\varphi_{i,t + 1}\| \big) \|x_{i,t} - z_{i,t + 1}\| \nonumber\\
		&\le \textstyle \frac{2G_s^2\gamma_0}{\eta\gamma_t^2} + \frac{2G_h^2\gamma_0}{\eta\gamma_t^2}\|\varphi_{i,t + 1}\|^2 + \frac{\eta \|x_{i,t} - z_{i,t + 1}\|^2}{4\gamma_0}.
	\end{align}
	
	Recalling the definition of $\Delta_{i,t}(\hat{x}_t)$ and $\alpha_t \gamma_t = \gamma_0$, we obtain
	\begin{align}
		&\textstyle\sum_{t = 1}^T \frac{\Delta_{i,t}(\hat{x}_t)}{\gamma_t} \le \frac{\mathcal{D}_\psi(\hat{x}_1, x_{i,1})}{\gamma_0} \nonumber\\
		&\quad \textstyle + \sum_{t=1}^{T} \frac{1}{\gamma_0} \left| \mathcal{D}_\psi(\hat{x}_{t+1}, x_{i,t+1}) - \mathcal{D}_\psi(\hat{x}_t, x_{i,t+1}) \right| .
	\end{align}
	
	Similar to (\ref{eq:penalty_cancel}) and noting $\gamma_0 \le {\eta}/{4G_h^2}$, we have
	\begin{align}\label{XXX70}
		& \textstyle \frac{2G_h^2\gamma_0}{\eta\gamma_t^2}\|\varphi_{i,t + 1}\|^2 - \frac{\varphi_{i,t + 1}^T h_{i,t}(x_{i,t})}{2\gamma_t} \nonumber\\
		&= \textstyle \bigg( \frac{2G_h^2\gamma_0}{\eta} - \frac{1}{2} \bigg) \|[h_{i,t}(x_{i,t})]_+\|^2 \le 0.
	\end{align}
As a result, combining (\ref{Deee})--(\ref{XXX70}),  we arrive at the final bound~\eqref{Lemma5-eq2}.
\end{proof}
}

\begin{lemma}\label{lem.33}
Under Assumptions~\ref{ass:outer} and~\ref{ass:inner}, consider the inner iteration~(\ref{estY}) in Algorithm~\ref{algo1} with $\beta \leq {2}/({\mu_g + L_{g,y}})$, we have
	\begin{flalign}\label{lemmYY}
		&\textstyle\left\|{y}_{i,t+1} - {y}_{i,t}^*({x}_{i,t})\right\|   \leq (1 - \beta\rho)\left\|{y}_{i,t} - {y}_{i,t-1}^*({x}_{i,t-1})\right\| \nonumber\\
		&\quad \textstyle + (1 - \beta\rho)\left\|{y}_{i,t-1}^*({x}_{i,t-1}) - {y}_{i,t}^*({x}_{i,t-1})\right\| \nonumber\\
		&\quad \textstyle + L_y (1 - \beta\rho)\left\|{x}_{i,t} - {x}_{i,t-1}\right\|,
	\end{flalign}
	where  $\rho = {\mu_g L_{g,y}}/({\mu_g + L_{g,y}})$.
\end{lemma}

\ifbool{submitA}{
	\begin{proof}
		The proof is given in Appendix A of the arXiv version \cite{Mypaper} and is omitted here due to space limitation.
	\end{proof} 
}{
\begin{proof}
	From (\ref{estY}), we have
	\begin{flalign}
		&\textstyle\left\|{y}_{i,t+1} - {y}_{i,t}^*({x}_{i,t})\right\|^2 \nonumber\\
		&\textstyle = \left\|{y}_{i,t} - {y}_{i,t}^*({x}_{i,t})\right\|^2 + \beta^2\left\|\nabla_{y} g_{i,t}({x}_{i,t}, {y}_{i,t})\right\|^2 \nonumber\\
		& \textstyle \quad- 2\beta\left\langle \nabla_{y} g_{i,t}({x}_{i,t}, {y}_{i,t}), {y}_{i,t} - {y}_{i,t}^*({x}_{i,t}) \right\rangle. \label{A.7_1_norm}
	\end{flalign}
	By  Assumption~\ref{ass:inner}, we have
	\begin{flalign}
		&\textstyle \left\langle \nabla_{y} g_{i,t}({x}_{i,t}, {y}_{i,t}), {y}_{i,t} - {y}_{i,t}^*({x}_{i,t}) \right\rangle \nonumber\\
		&\textstyle \geq \rho \left\|{y}_{i,t} - {y}_{i,t}^*({x}_{i,t})\right\|^2 + \frac{1}{\mu_g + L_{g,y}}\left\|\nabla_{y} g_{i,t}({x}_{i,t}, {y}_{i,t})\right\|^2. \label{A.7_2_norm}
	\end{flalign}
	Substitute \eqref{A.7_2_norm} into \eqref{A.7_1_norm} yields
	\begin{flalign}
		&\textstyle \left\|{y}_{i,t+1} - {y}_{i,t}^*({x}_{i,t})\right\|^2 \leq (1 - 2\beta\rho)\left\|{y}_{i,t} - {y}_{i,t}^*({x}_{i,t})\right\|^2 \nonumber\\
		&\textstyle \quad - \left(\frac{2\beta}{\mu_g + L_{g,y}} - \beta^2\right) \left\|\nabla_{y} g_{i,t}({x}_{i,t}, {y}_{i,t})\right\|^2. \label{A.7_3_norm}
	\end{flalign}
	Notice from (\ref{estY}) that $\left\|{y}_{i,t+1} - {y}_{i,t}\right\|^2 = \beta^2\left\|\nabla_{y} g_{i,t}({x}_{i,t}, {y}_{i,t})\right\|^2$.  From \eqref{A.7_3_norm}, since ${2}/{(\beta(\mu_g + L_{g,y}))} - 1 \geq 0$,  we obtain
	\begin{flalign}
		&\textstyle \left\|{y}_{i,t+1} - {y}_{i,t}^*({x}_{i,t})\right\|^2  \leq (1 - 2\beta\rho)\left\|{y}_{i,t} - {y}_{i,t}^*({x}_{i,t})\right\|^2\nonumber\\
		&\textstyle \leq (1-\beta\rho)^2 \| {y}_{i,t} - {y}_{i,t}^*({x}_{i,t})\|^2. \label{A.7_4_norm}
	\end{flalign}
	From \eqref{A.7_4_norm}, it yields
	 \begin{flalign}
	 	\left\|{y}_{i,t+1} - {y}_{i,t}^*({x}_{i,t})\right\| 
	 	\leq (1 - \beta\rho)\left\|{y}_{i,t} - {y}_{i,t}^*({x}_{i,t})\right\|. \label{A.7_5_combined}
	 \end{flalign}
	Moreover, we have
	\begin{flalign}
		& \textstyle\left\|{y}_{i,t} - {y}_{i,t}^*({x}_{i,t})\right\| \nonumber\\
		& \textstyle \leq \left\|{y}_{i,t} - {y}_{i,t-1}^*({x}_{i,t-1})\right\| + \left\|{y}_{i,t-1}^*({x}_{i,t-1}) - {y}_{i,t}^*({x}_{i,t-1})\right\| \nonumber\\
		& \quad + \left\|{y}_{i,t}^*({x}_{i,t-1}) - {y}_{i,t}^*({x}_{i,t})\right\|. \label{A.7_5_norm}
	\end{flalign}
	From (\ref{eq:lip_y_star}), we have
	\begin{flalign}
		\textstyle \left\|{y}_{i,t}^*({x}_{i,t-1}) - {y}_{i,t}^*({x}_{i,t})\right\| \leq L_y \left\|{x}_{i,t} - {x}_{i,t-1}\right\|. \label{A.7_6_norm}
	\end{flalign}
	
	Substitute \eqref{A.7_6_norm} and \eqref{A.7_5_norm}  into \eqref{A.7_5_combined}, we obtain (\ref{lemmYY}).
\end{proof}
}

\begin{lemma}\label{lem:v_bound_dist_single_norm}
	Under Assumptions~\ref{ass:outer} and~\ref{ass:inner}, consider~\eqref{eq:w_update}~and \eqref{eq:v_update} in Algorithm~\ref{algo1} with $\lambda \leq 1/L_{g,y}$, we have
	\begin{align}
		&\|{v}_{i,t+1} - u_{i,t}^*\| \leq \rho_v\|{v}_{i,t} - u_{i,t-1}^*\| \nonumber \\
		&\quad + \tilde{C}_{xy} \big(\|{x}_{i,t-1} - {x}_{i,t}\| + \|{y}_{i,t} - {y}_{i,t+1}\|\big) \nonumber \\
		&\quad + \tilde{C}_g \|\nabla_{{y}{y}}^2 g_{i,t-1}({x}_{i,t}, {y}_{i,t+1}) - \nabla_{{y}{y}}^2 g_{i,t}({x}_{i,t}, {y}_{i,t+1})\| \nonumber \\
		&\quad + \tilde{C}_f \|\nabla_{y} f_{i,t-1}({x}_{i,t}, {y}_{i,t+1}) - \nabla_{y} f_{i,t}({x}_{i,t}, {y}_{i,t+1})\|, \label{eq:lem_v_bound_final}
	\end{align}
	where $u_{i,t}^* =v_{i,t}(x_{i,t},y_{i,t+1})= [\nabla_{yy}^2 g_{i,t}(x_{i,t}, y_{i,t+1})]^{-1} \nabla_y f_{i,t}(x_{i,t}, y_{i,t+1}) $, $\rho_v = 1-{\lambda\mu_g}/{2} \in (0, 1)$, $\tilde{C}_f = {\rho_v}/{\mu_g}$,  $\tilde{C}_g ={\rho_v C_{f,y}}/{\mu_g^2}$,  $\tilde{C}_{xy} = \rho_v L_v$.  
\end{lemma}

\ifbool{submitA}{
	\begin{proof}
		The proof is given in Appendix A of the arXiv version \cite{Mypaper} and is omitted here due to space limitation.
	\end{proof} 
}{
\begin{proof}
Based on the updates \eqref{eq:w_update} and \eqref{eq:v_update}, Assumptions~\ref{ass:outer} and~\ref{ass:inner}, and the non-expansiveness of the projection, we have
\begin{align}
	\|v_{i,t+1} - u_{i,t}^*\|^2 &\leq \|w_{i,t+1} - u_{i,t}^*\|^2 \nonumber \\
	&\leq (1 - \lambda\mu_g) \|v_{i,t} - u_{i,t}^*\|^2. \label{eq:contraction_sq}
\end{align}

Furthermore, by the triangle inequality and $\rho_v \ge \sqrt{1-\lambda\mu_g}$, it follows from~(\ref{eq:contraction_sq}) that
	\begin{align}
		&\|v_{i,t+1} - u_{i,t}^*\| \leq \|w_{i,t+1} - u_{i,t}^*\| \leq \rho_v \|v_{i,t} - u_{i,t}^*\| \nonumber \\
		&\quad\leq \rho_v \|v_{i,t} - u_{i,t-1}^*\| + \rho_v \|u_{i,t-1}^* - u_{i,t}^*\|. \label{eq:v_bound_triangle}
	\end{align}
 We decompose the optimal shift $\|u_{i,t-1}^* - u_{i,t}^*\|$ into spatial ($S_1$) and temporal ($S_2$) variations
	\begin{align}
		&\|u_{i,t-1}^* - u_{i,t}^*\| \leq \underbrace{\|v_{i,t-1}({x}_{i,t-1}, {y}_{i,t}) - v_{i,t-1}({x}_{i,t}, {y}_{i,t+1})\|}_{S_1} \nonumber\\
		&\qquad + \underbrace{\|v_{i,t-1}({x}_{i,t}, {y}_{i,t+1}) - v_{i,t}({x}_{i,t}, {y}_{i,t+1})\|}_{S_2}. \label{eq:shift_split}
	\end{align}
	For $S_1$, by directly leveraging the $L_v$-Lipschitz continuity of $v_{i,t-1}(x, y)$ established in (\ref{eq:lip_v}), we obtain
	\begin{align}
		S_1 \leq L_v \left(\|{x}_{i,t-1} - {x}_{i,t}\| + \|{y}_{i,t} - {y}_{i,t+1}\|\right). \label{eq:S1_bound}
	\end{align}
	For $S_2$, simple calculations yield
\begin{equation}
	\begin{aligned}
		\textstyle S_2 & \textstyle= \Big\|\big[\nabla_{yy}^2 g_{i,t-1}({x}_{i,t}, {y}_{i,t+1})\big]^{-1}\nabla_{y} f_{i,t-1}({x}_{i,t}, {y}_{i,t+1}) \\
		&\quad  \textstyle - \big[\nabla_{yy}^2 g_{i,t}({x}_{i,t}, {y}_{i,t+1})\big]^{-1}\nabla_{y} f_{i,t}({x}_{i,t}, {y}_{i,t+1})\Big\| \\
		&\textstyle \leq \frac{C_{f,y}}{\mu_g^2}\left\|\nabla_{yy}^2 g_{i,t-1}({x}_{i,t}, {y}_{i,t+1}) - \nabla_{yy}^2 g_{i,t}({x}_{i,t}, {y}_{i,t+1})\right\| \\
		&\quad \textstyle + \frac{1}{\mu_g}\left\|\nabla_{y} f_{i,t-1}({x}_{i,t}, {y}_{i,t+1}) - \nabla_{y} f_{i,t}({x}_{i,t}, {y}_{i,t+1})\right\|. \label{eq:S2_bound}
	\end{aligned}
\end{equation}
 
 Substituting \eqref{eq:S1_bound} and \eqref{eq:S2_bound} into \eqref{eq:shift_split}, and multiplying by the scaling parameter $\rho_v$, we arrive at \eqref{eq:lem_v_bound_final}. 
\end{proof}
}

\begin{lemma}\label{lem:v_bound_dist_multi_norm}
	Under Assumptions~\ref{ass:outer} and~\ref{ass:inner}, given $\lambda < 1/L_{g,y}$ and $M$ HvP iterations as defined in Algorithm~\ref{alg.2}, define $\rho_H=1-\lambda \mu_g \in (0,1)$. Then,
	\begin{align}
	\textstyle \|v_{i,t+1} - v_{i,t}^*(x_{i,t})\| \leq& \textstyle (r_v+\frac{C_{f,y}}{\mu_g})\rho_H^{M/2}  \nonumber\\
	&\textstyle + L_v\|y_{i,t+1} - y_{i,t}^*(x_{i,t})\|. \label{eq:v_bound_dist_res_norm}
	\end{align}
\end{lemma}  

\ifbool{submitA}{
	\begin{proof}
		The proof is given in Appendix A of the arXiv version \cite{Mypaper} and is omitted here due to space limitation.
	\end{proof} 
}{
\begin{proof}
	Define  $\hat{v}_{i,t} \triangleq v_{i,t}(x_{i,t}, y_{i,t+1})$. By the triangle inequality, we have
	\begin{align}
		\textstyle \|v_{i,t+1} - v_{i,t}^*(x_{i,t})\| \leq \|v_{i,t+1} - \hat{v}_{i,t}\| + \|\hat{v}_{i,t} - v_{i,t}^*(x_{i,t})\|. \label{eq:proof_triangle}
	\end{align}
	
 For the first term in (\ref{eq:proof_triangle}), according to Assumption~\ref{ass:inner}, Lemma~\ref{lemVM} and the non-expansiveness of the projection \eqref{eq:v_proj_loop}, we have
	\begin{align}
		\textstyle  \|v_{i,t+1} - \hat{v}_{i,t}\| \leq \rho_H^{M/2}\|v_{i,t+1}^0 - \hat{v}_{i,t}\| \leq (r_v+\frac{C_{f,y}}{\mu_g})\rho_H^{M/2}. \label{eq:iter_error}
	\end{align}
For the second term in (\ref{eq:proof_triangle}), leveraging (\ref{eq:lip_v}), we have
	\begin{align}
		&\textstyle  \|\hat{v}_{i,t} - v_{i,t}^*(x_{i,t})\| = \|v_{i,t}(x_{i,t}, y_{i,t+1}) - v_{i,t}(x_{i,t}, y_{i,t}^*(x_{i,t}))\| \nonumber \\
		&\textstyle  \leq L_v\|y_{i,t+1} - y_{i,t}^*(x_{i,t})\|.\label{eq:shift_error}
	\end{align}
	Combining~\eqref{eq:proof_triangle}--\eqref{eq:shift_error}, we conclude that \eqref{eq:v_bound_dist_res_norm} holds.
\end{proof}
}

\subsection{Proof of Theorem~\ref{thm1}}\label{BBBB}
\begin{proof} 
(i) By the convexity of the outer-level function $f_t$, for any $i \in [N]$ and $t \in [T]$, we have
\begin{align*}
\textstyle f_t(x_{i,t}) - f_t(x_t^*) \le \langle \nabla f_t(x_{i,t}), x_{i,t} - x_t^* \rangle.
\end{align*}
Therefore, 
\begin{align*}
\textstyle  \mathbb{E}[\mathrm{Net}\text{-}\mathrm{Reg}(T)] \le \frac{1}{N} \sum_{i=1}^N \sum_{t=1}^T \mathbb{E}[\langle \nabla f_t(x_{i,t}), x_{i,t} - x_t^* \rangle].
\end{align*}
To bound the right-hand side, for any $\{\hat{x}_t\}_{t=1}^{T+1}$ satisfying the conditions of Lemma~\ref{lem.55}, using Lemma~\ref{lem.11} and the diameter bound \(\|x-y\|\le d(\mathcal X)\) for \(x,y\in\mathcal X\), we obtain
\begin{align}\label{KK1}
	&\textstyle\frac{1}{N}\sum_{i=1}^{N}\sum_{t=1}^{T}\mathbb{E}\left[\langle\nabla f_{t}(x_{i,t}),x_{i,t}-\hat{x}_t\rangle\right] \nonumber\\
	&\textstyle \le \frac{1}{N}\sum_{i=1}^{N}\sum_{t=1}^{T}\mathbb{E}\left[\langle\tilde{\nabla}f_{i,t}(x_{i,t},y_{i,t+1},v_{i,t+1}),x_{i,t}-\hat{x}_t\rangle\right] \nonumber\\
	& \textstyle +\underbrace{\textstyle  \frac{d(\mathcal{X})}{N}\sum_{i=1}^{N}\sum_{t=1}^{T}\mathbb{E}\left[\varUpsilon_{i,t}\right]}_{S_3}+\underbrace{ \textstyle  \frac{2G_s}{N}\sum_{i=1}^{N}\sum_{t=1}^{T}\mathbb{E}\left[\|x_{i,t}-\bar{x}_{t}\|\right]}_{S_4}\nonumber\\
	& \textstyle +\underbrace{ \textstyle \frac{d(\mathcal{X})L}{N^{2}}\sum_{i=1}^{N}\sum_{j=1}^{N}\sum_{t=1}^{T}\mathbb{E}\left[\|x_{i,t}-x_{j,t}\|\right]}_{S_5},
\end{align}
where $\bar{x}_t=1/N \sum_{j=1}^{N} x_{j,t}$.

 Let 
\begin{equation}\label{varUp}
	\textstyle  \varUpsilon _{i,t}=\|\nabla f_{i,t}(x_{i,t},y_{i,t}^*(x_{i,t})) - \tilde{\nabla}f_{i,t}(x_{i,t},y_{i,t+1},v_{i,t+1})\|.
\end{equation}
From (\ref{eq:hypergradient}), (\ref{HuGradi}) and (\ref{varUp}), we have
\begin{align}\label{KKey}
	&\textstyle	\varUpsilon_{i,t}  = \big\|\nabla_x f_{i,t}(x_{i,t},y_{i,t}^*(x_{i,t})) - \nabla_x f_{i,t}(x_{i,t},y_{i,t+1}) \nonumber\\
	& \quad + \textstyle  \nabla_{xy}^2 g_{i,t}(x_{i,t},y_{i,t+1})v_{i,t+1} \nonumber\\
	& \quad \textstyle - \nabla_{xy}^2 g_{i,t}(x_{i,t},y_{i,t}^*(x_{i,t}))v_{i,t}^*(x_{i,t})\big\| \nonumber\\
	& \textstyle \le \big\|\nabla_x f_{i,t}(x_{i,t},y_{i,t}^*(x_{i,t})) - \nabla_x f_{i,t}(x_{i,t},y_{i,t+1})\big\| \nonumber\\
	& \quad \textstyle + \big\|\nabla_{xy}^2 g_{i,t}(x_{i,t},y_{i,t+1})v_{i,t+1} \nonumber\\
	& \quad \textstyle - \nabla_{xy}^2 g_{i,t}(x_{i,t},y_{i,t}^*(x_{i,t}))v_{i,t}^*(x_{i,t})\big\|\nonumber\\
	& \textstyle \le L_{f,x}\big\|y_{i,t+1} - y_{i,t}^*(x_{i,t})\big\|  + \big\|\nabla_{xy}^2 g_{i,t}(x_{i,t},y_{i,t+1})v_{i,t+1} \nonumber\\
	& \quad \textstyle - \nabla_{xy}^2 g_{i,t}(x_{i,t},y_{i,t+1})v_{i,t}^*(x_{i,t})\big\| \nonumber\\
	& \quad \textstyle  + \big\|\nabla_{xy}^2 g_{i,t}(x_{i,t},y_{i,t+1})v_{i,t}^*(x_{i,t}) \nonumber\\
	& \quad \textstyle  - \nabla_{xy}^2 g_{i,t}(x_{i,t},y_{i,t}^*(x_{i,t}))v_{i,t}^*(x_{i,t})\big\| \nonumber\\
	& \textstyle  \le \left(L_{f,x} + \frac{C_{f,y}L_{g,xy}}{\mu_g}\right)\big\|y_{i,t+1} - y_{i,t}^*(x_{i,t})\big\| \nonumber\\
	& \quad \textstyle + C_{g,xy}\big\|v_{i,t+1} - u_{i,t}^*\big\|+ C_{g,xy}\big\|u_{i,t}^*- v_{i,t}^*(x_{i,t})\big\|\nonumber\\
	& \textstyle  \le \Phi_{i,t+1},
\end{align}
where $\Phi_{i,t+1}=\kappa_F \big\|y_{i,t+1} - y_{i,t}^*(x_{i,t})\big\|  + C_{g,xy}\big\|v_{i,t+1} - u_{i,t}^*\big\|$, with $\kappa_F=L_{f,x} + L_{g,xy}C_{f,y}/\mu_g + C_{g,xy}L_v + 2 C_{g,xy} \tilde{C}_{xy} L_{g,y}/\rho$.  Furthermore, substituting \eqref{lemmYY} and \eqref{eq:lem_v_bound_final} into \eqref{KKey}, and utilizing $\| y_{i,t+1} - y_{i,t} || \le \beta L_{g,y} \| y_{i,t} - y_{i,t}^*(x_{i,t}) \|$, we obtain
\begin{align}\label{Fkappa}
	& \textstyle  \Phi_{i,t+1} \le \kappa_F(1 - \beta\rho) \left\|y_{i,t} - y_{i,t-1}^*(x_{i,t-1})\right\| \nonumber\\
	&\qquad \qquad\textstyle  + C_{g,xy}\rho_v \left\|v_{i,t} - u_{i,t-1}^*\right\| \nonumber\\
	&\quad + \big( \kappa_F L_y(1 - \beta\rho) + C_{g,xy}\tilde{C}_{xy} \big) \left\|x_{i,t-1} - x_{i,t}\right\| \nonumber\\
	&\quad \textstyle  + \kappa_F(1 - \beta\rho) \left\|y_{i,t-1}^*(x_{i,t-1}) - y_{i,t}^*(x_{i,t-1})\right\| \nonumber\\
	&\quad \textstyle  + C_{g,xy}\tilde{C}_g \left\|\nabla_{yy}^2 g_{i,t-1}(x_{i,t}, y_{i,t+1}) - \nabla_{yy}^2 g_{i,t}(x_{i,t}, y_{i,t+1})\right\| \nonumber\\
	&\quad \textstyle  + C_{g,xy}\tilde{C}_f \left\|\nabla_y f_{i,t-1}(x_{i,t}, y_{i,t+1}) - \nabla_y f_{i,t}(x_{i,t}, y_{i,t+1})\right\| \nonumber\\
	&\quad \textstyle + C_{g,xy}\tilde{C}_{xy} \beta L_{g,y} \left\|y_{i,t} - y_{i,t}^*(x_{i,t})\right\|.
\end{align}

Utilizing the definition of $\kappa_F$ which ensures $\kappa_F (1 - \beta \rho) + C_{g,xy} \tilde{C}_{xy} \beta L_{g,y} \le \kappa_F \left(1 -  \beta \rho/2\right)$ and setting $\bar{\rho} = \max\{1 - \beta\rho/2,  \rho_v\}\in (0,1)$, for $t\ge 2$  we obtain
\begin{align}\label{AA98}
	\textstyle &\Phi_{i,t+1}
	\le \textstyle  \bar{\rho}\Phi_{i,t}+ C_x  \|x_{i,t-1} - x_{i,t}\| \nonumber\\
	& \textstyle \qquad  + C_{\alpha} \sup_{x} \|y_{i,t-1}^*(x) - y_{i,t}^*(x)\| \nonumber\\
	& \textstyle   \qquad  + C_g^H  \sup_{x,y} \|\nabla_{yy}^2 g_{i,t-1}(x, y) - \nabla_{yy}^2 g_{i,t}(x, y)\| \nonumber\\
	&\textstyle \qquad    + C_f^g \sup_{x,y} \|\nabla_y f_{i,t-1}(x, y) - \nabla_y f_{i,t}(x, y)\|, 
\end{align}
where $C_{x} = \kappa_F L_y(1 - \beta\rho) + C_{g,xy}\tilde{C}_{xy}(1+\beta L_{g,y}L_y)$, $C_{\alpha} = \kappa_F(1 - \beta\rho/2)$, $C_g^H = C_{g,xy}\tilde{C}_g$ and $C_f^g = C_{g,xy}\tilde{C}_f$.

By (\ref{KKey}) and (\ref{AA98}), we further obtain 
\begin{align}\label{eq:cumulative_tracking_error}
\textstyle \sum_{t=1}^T \varUpsilon_{i,t} \le& \textstyle \frac{\Phi_{i,2}}{1-\bar{\rho}} + \frac{C_x}{1-\bar{\rho}} \sum_{t=1}^{T-1} \|x_{i,t} - x_{i,t+1}\| + \frac{C_\alpha}{1-\bar{\rho}} H_{1,T}\nonumber\\
& \textstyle+ \frac{C_g^H}{1-\bar{\rho}} E_{1,yy,T}^g + \frac{C_f^g}{1-\bar{\rho}} E_{1,y,T}^f . 
\end{align}

From (\ref{eq:cumulative_tracking_error}), we have
\begin{equation}\label{KK2}
	\begin{aligned}
		& \textstyle S_3 \le \frac{d(\mathcal{X})}{N(1-\bar{\rho})} \sum_{i=1}^{N} \mathbb{E}\left[\Phi_{i,2}\right] + \frac{d(\mathcal{X})C_{\alpha}}{1-\bar{\rho}} H_{1,T}\\
		&\qquad  \textstyle  +  \textstyle \underbrace{ \textstyle \frac{d(\mathcal{X})C_x}{N(1-\bar{\rho})} \sum_{i=1}^{N}\sum_{t=1}^{T-1} \mathbb{E}\left[\|x_{i,t} - x_{i,t+1}\|\right] }_{S_6}\\
		&\qquad \textstyle  + \frac{d(\mathcal{X})C_g^H}{1-\bar{\rho}} E_{1,yy,T}^g + \frac{d(\mathcal{X})C_f^g}{1-\bar{\rho}} E_{1,y,T}^f.
	\end{aligned}
\end{equation}

Moreover, from (\ref{Algorithm1-eq3}), it holds that
\begin{align}\label{AA96}
	&\textstyle \|x_{i,t+1} - x_{i,t}\| \nonumber\\
	& \textstyle \le \sum_{j=1}^N [W]_{ij} \|x_{j,t} - x_{i,t}\| + \sum_{j=1}^N [W]_{ij} \|\varepsilon_{j,t}^z\|.
\end{align}
Summing (\ref{AA96}) over $i \in [N]$ yields
\begin{align}
	&\textstyle	\sum_{i=1}^N \|x_{i,t+1} - x_{i,t}\| \nonumber\\
	&\textstyle \le \sum_{i=1}^N \sum_{j=1}^N \|x_{i,t} - x_{j,t}\| + \sum_{i=1}^N \|\varepsilon_{i,t}^z\|, \label{eq:variation_consensus_bridge}
\end{align}
where the inequality holds since Assumption~\ref{ass:network_fixed}.

By utilizing $\|x_{i,t} - \bar{x}_t\| \le 1/N \sum_{j=1}^N \|x_{i,t} - x_{j,t}\|$,  (\ref{Cons11}) and (\ref{eq:variation_consensus_bridge}), we have
\begin{align}
\textstyle	S_4 + S_5 + S_6 \le \tilde{\psi}_1 + \frac{\tilde{\psi}_2}{N} \sum_{i=1}^{N}\sum_{t=1}^{T} \mathbb{E}\left[\|\varepsilon_{i,t}^z\|\right], \label{eq:S346_bound}
\end{align}
where $\tilde{\psi}_1 = ( d(\mathcal{X})L + 2G_s + d(\mathcal{X})C_x N / (1-\bar{\rho}) ) \vartheta_1$ and $\tilde{\psi}_2 = ( d(\mathcal{X})L + 2G_s )\vartheta_2 + d(\mathcal{X})C_x(N\vartheta_2 + 1) / (1-\bar{\rho})$.

By (\ref{Algorithm1-eq1}), Assumption~\ref{ass.cpmpress} and the non-expansiveness of $P_\mathcal{X}$,  we have $\mathbb{E}[\Vert{} z_{i, t+1} - \hat{z}_{i, t+1} \Vert{}^2] \le C s_{t+1}^2$.  Combining this estimate with the triangle inequality and Young's inequality gives
\begin{align}
	& \textstyle \frac{\tilde{\psi}_2}{N} \sum_{t=1}^{T} \sum_{i=1}^{N} \mathbb{E}\left[\|\varepsilon_{i,t}^z\|\right] \nonumber\\
	&\textstyle \le \frac{\tilde{\psi}_2}{N} \sum_{t=1}^{T} \sum_{i=1}^{N} \mathbb{E}\left[\|\hat{z}_{i,t+1} - z_{i,t+1}\|\right] \nonumber\\
	&\quad \textstyle + \frac{\tilde{\psi}_2}{N} \sum_{t=1}^{T} \sum_{i=1}^{N} \mathbb{E}\left[\|z_{i,t+1} - x_{i,t}\|\right] \nonumber\\
	&\le \textstyle \frac{1}{N} \sum_{t=1}^{T} \sum_{i=1}^{N} \mathbb{E}\left[\frac{\tilde{\psi}_2^2 \alpha_t}{\eta} + \frac{\eta\|z_{i,t+1} - x_{i,t}\|^2}{4\alpha_t}\right]\nonumber\\
	&\quad \textstyle +\sqrt{C}\tilde{\psi}_2 \sum_{t=1}^{T} s_{t+1} . \label{eq:epsilon_bound}
\end{align}

From (\ref{canshu}), we have
\begin{align}
\textstyle	\sum_{t = 1}^T \alpha_t & \textstyle \le \frac{\alpha_0 T^{1 - \theta_1}}{1 - \theta_1}, \quad \sum_{t = 1}^T s_t \le s_0 \left(  \frac{\theta_1+1}{\theta_1} \right), \label{eq:alpha_sum} \\
\textstyle	\sum_{t = 1}^T \frac{s_t}{\alpha_t} & \textstyle \le \frac{s_0}{\alpha_0} (1 + \ln T) , \quad \sum_{t = 1}^T \frac{s_t^2}{\alpha_t} \le \frac{s_0^2}{\alpha_0} \left( \frac{\theta_1 + 2}{\theta_1 + 1} \right). \label{eq:s_alpha_sum}
\end{align}

Besides, from (\ref{Klip})--(\ref{KKbound}), it holds that
\begin{align}\label{quer}
	\textstyle \frac{1}{N} & \textstyle \sum_{t=1}^T \sum_{i=1}^N \Delta_{i, t}\left(\hat{x}_t\right) \nonumber\\
	\leq &  \textstyle \frac{1}{N} \sum_{i=1}^N\left(\frac{1}{\alpha_1} \mathcal{D}_\psi\left(\hat{x}_1, x_{i, 1}\right)-\frac{1}{\alpha_T} \mathcal{D}_\psi\left(\hat{x}_{T+1}, x_{i, T+1}\right)\right)\nonumber \\
	&\textstyle +\left(\frac{1}{\alpha_T}-\frac{1}{\alpha_1}\right) Kd(\mathcal{X}) \nonumber\\
	&\textstyle +\frac{1}{N} \sum_{i=1}^N\sum_{t=1}^T \frac{1}{\alpha_t}\left(\mathcal{D}_\psi\left(\hat{x}_{t+1}, x_{i, t+1}\right)-\mathcal{D}_\psi\left(\hat{x}_t, x_{i, t+1}\right)\right) \nonumber\\
	\leq &\textstyle\frac{1}{N} \sum_{i=1}^N \sum_{t=1}^T \frac{1}{\alpha_t}\left(\mathcal{D}_\psi\left(\hat{x}_{t+1}, x_{i, t+1}\right)-\mathcal{D}_\psi\left(\hat{x}_t, x_{i, t+1}\right)\right) \nonumber\\
	& \textstyle +\frac{Kd(\mathcal{X}) }{\alpha_T}\nonumber\\
	\leq &\textstyle K \sum_{t=1}^T \frac{1}{\alpha_t} \|\hat{x}_{t+1}-\hat{x}_{t}\|+\frac{Kd(\mathcal{X}) }{\alpha_T}. 
\end{align}

Combining~\eqref{Lemma5-eq1} \eqref{KK1}, \eqref{KK2} and  \eqref{eq:S346_bound}--\eqref{quer}, and choosing $\hat{x}_t=x_t^*$ for $t \in [T]$ and $\hat{x}_{T+1} = x_T^*$,  we have
\begin{align}\label{eq:Net_Reg_Final}
	&\textstyle  \mathbb{E}[\mathrm{Net\text{-}Reg}(T)] \nonumber\\
	&\le \textstyle  \tilde{\psi}_3+\frac{(2G_s^2 + \tilde{\psi}_2^2)\alpha_0}{\eta(1 - \theta_1)} T^{1 - \theta_1} + \left( \frac{K d(\mathcal{X}) }{\alpha_0} + \frac{K}{\alpha_0} P_T \right) T^{\theta_1} \nonumber\\
	&\quad \textstyle  + \frac{d(\mathcal{X})L_{\psi} \sqrt{C}s_0}{\alpha_0} (1 + \ln T) + \frac{d(\mathcal{X})C_{\alpha}}{1-\bar{\rho}} H_{1,T} \nonumber\\
	&\quad \textstyle  + \frac{d(\mathcal{X})C_g^H}{1-\bar{\rho}} E_{1,yy,T}^g + \frac{d(\mathcal{X})C_f^g}{1-\bar{\rho}} E_{1,y,T}^f,
\end{align}
where $\tilde{\psi}_3 = \tilde{\psi}_1 + L_{\psi}  C (s_0^2 / \alpha_0) ((\theta_1 + 2) / (\theta_1 + 1)) + \sqrt{C}\tilde{\psi}_2 s_0 (1 + 1 / \theta_1) + (d(\mathcal{X}) / (N(1-\bar{\rho}))) \sum_{i=1}^{N}\mathbb{E} [\Phi_{i,2}]$.
Thus, \eqref{thm1ret} holds.

%$h_t(x) = \text{col}(h_{1,t}(x), \dots, h_{N,t}(x))$

(ii) 
From the non-expansiveness of the projection $[\cdot]_+$ and Assumption~\ref{ass:constraint} (ii), for any $i,j\in [N]$ we have
\begin{align}
	&\|[h_{j,t}(x_{i,t})]_+\|^2 \nonumber\\
	&\le 2\|[h_{j,t}(x_{j,t})]_+\|^2 + 2\|[h_{j,t}(x_{i,t})]_+ - [h_{j,t}(x_{j,t})]_+\|^2 \nonumber \\
	&\le 2\|[h_{j,t}(x_{j,t})]_+\|^2 + 2G_h^2 \|x_{i,t} - x_{j,t}\|^2. \label{Lemma6-proof-ii-eq1}
\end{align}

Summing \eqref{Lemma6-proof-ii-eq1} over $i\in [N]$, taking the expectation, and averaging over $i \in [N]$, we obtain
\begin{align}
	&\textstyle \frac{1}{N}\sum_{i=1}^N \sum_{t=1}^T \mathbb{E}\big[\|[h_t(x_{i,t})]_+\|^2\big] \nonumber \\
	&\textstyle \le \frac{2}{N}\sum_{i=1}^N \sum_{j=1}^N \sum_{t=1}^T \mathbb{E}\big[\|[h_{j,t}(x_{j,t})]_+\|^2\nonumber\\
	&\textstyle\quad  + G_h^2 \|x_{i,t} - x_{j,t}\|^2\big] \nonumber \\
	&\textstyle \le \tilde{\psi}_4 + 2\sum_{i=1}^N \sum_{t=1}^T \mathbb{E}\big[\|[h_{i,t}(x_{i,t})]_+\|^2\big] \nonumber \\
	&\quad + \textstyle \tilde{\psi}_5\sum_{i=1}^N \sum_{t=1}^T \mathbb{E}\big[\|z_{i,t+1} - x_{i,t}\|^2\big]  + \tilde{\psi}_6\sum_{t=1}^T s_t^2\nonumber \\
		&\textstyle \le\tilde{\psi}_4 + \tilde{\psi}_6\sum_{t = 1}^T s_t^2+\tilde{\psi}_7\tilde{\Lambda}_T(\hat{x}_{[T+1]}) ,\quad  \forall \hat{x}_{[T+1]}\subset \Omega.
	 \label{Lemma6-proof-ii-eq3}
\end{align}
where $\tilde{\psi}_4=2G_h^2\vartheta_3$, $\tilde{\psi}_5= 4G_h^2{\vartheta}_4$,  $\tilde{\psi}_6=4N C G_h^2\vartheta_4$, $\tilde{\psi}_7 = {4\max\{1,\tilde{\psi}_5\}}/{\min\{1, \eta/(2\gamma_0)\}}$, the second inequality holds due to the non-increasing property of $\{s_t\}$, (\ref{Cons22}) and $\mathbb{E}[\Vert{} z_{i, t+1} - \hat{z}_{i, t+1} \Vert{}^2] \le C s_{t+1}^2$, and the last inequality is due to $\varphi_{i,t+1}^Th_{i,t}(\hat{x})\le 0$ for any $\hat{x}\in \Omega$ and \eqref{Lemma5-eq2}.

%$\sum_{i=1}^{n} \sum_{j=1}^{n} \sum_{t=1}^{T} \mathbb{E}[\|[h_{i,t}(x_{j,t})]_{+}\|^2] = \sum_{i=1}^{n} \sum_{t=1}^{T} \mathbb{E}[\|[h_{t}(x_{i,t})]_{+}\|^2]$

By the Cauchy--Schwarz inequality, we have
\begin{align}
	&\textstyle \bigg( \frac{1}{N}\sum_{i = 1}^N \sum_{t = 1}^T \mathbb{E}\big[\|[h_t(x_{i,t})]_+\|\big] \bigg)^2 \nonumber\\
	&\textstyle \le \frac{T}{N}\sum_{i = 1}^N \sum_{t = 1}^T \mathbb{E}\big[\|[h_t(x_{i,t})]_+\|^2\big]. \label{Lemma6-proof-ii-eq6}
\end{align}

From $\gamma_t = \gamma_0/\alpha_t$ and the definition of the parameters in (\ref{canshu}), it holds that
\begin{align}
\textstyle	\sum_{t = 1}^T \frac{1}{\gamma_t} & \textstyle= \sum_{t = 1}^T \frac{\alpha_t}{\gamma_0} \le \frac{\alpha_0 T^{1 - \theta_1}}{(1 - \theta_1)\gamma_0}, \label{Theorem1-proof-ii-eq1}\\
\textstyle	\sum_{t = 1}^T \frac{\gamma_0}{\gamma_t^2} & \textstyle \le \sum_{t = 1}^T \frac{\alpha_t^2}{\gamma_0} \le \frac{\alpha_0^2 T^{1 - \theta_1}}{(1 - \theta_1)\gamma_0}, \label{Theorem1-proof-ii-eq2}\\
\textstyle	\sum_{t = 1}^T s_t &\textstyle = \sum_{t = 1}^T \frac{s_0}{t^{\theta_1 + 1}} \le s_0 \left( 1 + \frac{1}{\theta_1} \right). \label{Theorem1-proof-ii-eq3}
\end{align}

From (\ref{KKbound}), we have
\begin{align}\label{Theorem1-proof-ii-eq4}
\textstyle \frac{1}{\gamma_0}\sum_{i = 1}^N\mathbb{E} [{\mathcal{D}_\psi(\hat{x}, x_{i,1})}] \le\frac{NK d(\mathcal{X})}{\gamma_0}.
\end{align}

Combining the expression of $\tilde{\Lambda}_T(\hat{x}_{[T+1]})$ in (\ref{Lambda_def}) with (\ref{eq:alpha_sum}), \eqref{Theorem1-proof-ii-eq1}--\eqref{Theorem1-proof-ii-eq4}, and substituting it into \eqref{Lemma6-proof-ii-eq3} and \eqref{Lemma6-proof-ii-eq6}, we obtain
\begin{align}
	&\textstyle \big( \textstyle\frac{1}{N}\sum_{i = 1}^N \sum_{t = 1}^T \mathbb{E}\big[\|[h_t(x_{i,t})]_+\|\big] \big)^2 \nonumber\\
	&\textstyle\le \tilde{\psi}_4 T + \frac{\tilde{\psi}_7 NK d(\mathcal{X})}{\gamma_0} T  + \frac{2NG_s d(\mathcal{X}) \tilde{\psi}_7 \alpha_0}{(1 - \theta_1)\gamma_0} T^{2 - \theta_1} \nonumber\\
	&\quad \textstyle + \frac{2N G_s^2 \tilde{\psi}_7 \alpha_0^2}{\gamma_0 \eta (1 - \theta_1)} T^{2 - \theta_1} + \frac{N(\theta_1 + 1)}{\gamma_0\theta_1} d(\mathcal{X}) L_{\psi}\sqrt{C}\tilde{\psi}_7 s_0 T \nonumber\\
	&\quad \textstyle+ \frac{2(\theta_1 + 1)}{2\theta_1 + 1} \left(\frac{N}{\gamma_0} \tilde{\psi}_7 L_{\psi} C + \tilde{\psi}_6 \right) s_0^2 T. \label{eq:constraint_violation_bound}
\end{align}

Taking the square root of both sides of \eqref{eq:constraint_violation_bound}, it follows that \eqref{thm1Cvv11} holds.

(iii) By the non-expansiveness of the projection $[\cdot]_+$ and Assumption~\ref{ass:constraint} (ii), we have
\begin{align}
	&\textstyle \frac{1}{N}\sum_{j = 1}^N \sum_{t = 1}^T \mathbb{E}\big[\|[h_t(x_{j,t})]_+\|\big] \nonumber\\
	&\textstyle \le \frac{1}{N}\sum_{i = 1}^N \sum_{j = 1}^N \sum_{t = 1}^T \mathbb{E}\big[\|[h_{i,t}(x_{i,t})]_+\| \nonumber\\
	&\quad\textstyle + \|[h_{i,t}(x_{j,t})]_+ - [h_{i,t}(x_{i,t})]_+\|\big] \nonumber\\
	&\textstyle \le \frac{1}{N}\sum_{i = 1}^N \sum_{j = 1}^N \sum_{t = 1}^T \mathbb{E}\big[\|[h_{i,t}(x_{i,t})]_+\| + G_h\|x_{i,t} - x_{j,t}\|\big]. \label{Lemma6-proof-iii-eq1}
\end{align}

Applying \cite[Lemma 3]{Xinlei2025} to the mirror descent update \eqref{Algorithm1-eq6} and noting $\varphi_{i,t+1} = \gamma_t [h_{i,t}(x_{i,t})]_+$, we have
\begin{align}
	&\textstyle \|z_{i,t+1} - x_{i,t}\| \nonumber\\
	&\textstyle \le \frac{\alpha_t}{\eta} \|\tilde{\nabla}f_{i,t}(x_{i,t},y_{i,t+1},v_{i,t+1}) + \nabla h_{i,t}(x_{i,t})^T \varphi_{i,t+1}\| \nonumber\\
	&\textstyle \le \frac{\alpha_t G_s}{\eta} + \frac{G_h \gamma_0}{\eta} \|[h_{i,t}(x_{i,t})]_+\|, \label{eq:z_x_variation}
\end{align}
where the last inequality holds since $\alpha_t \gamma_t = \gamma_0$ and $\Vert{}\tilde{\nabla} f_{i,t} (x_{i,t}, y_{i,t+1}, v_{i,t+1})\Vert{} \le G_s$.

Substituting \eqref{eq:z_x_variation} into \eqref{Cons11}, and using $\mathbb{E}[\Vert{}\hat{z}_{i,t+1} - z_{i,t+1}\Vert{}] \le \sqrt{C} s_{t+1} \le \sqrt{C} s_t$, we obtain
\begin{align}
	&\textstyle\frac{1}{N}\sum_{i = 1}^N \sum_{j = 1}^N \sum_{t = 1}^T \mathbb{E}\big[\|x_{i,t} - x_{j,t}\|\big] \nonumber\\
	&\textstyle \le N\vartheta_1 + {\vartheta}_2\sum_{i = 1}^N \sum_{t = 1}^T \mathbb{E}\big[\|z_{i,t+1} - x_{i,t}\| \big] \nonumber\\
	&\quad \textstyle +  {\vartheta}_2\sum_{i = 1}^N \sum_{t = 1}^T\mathbb{E}\big[ \|\hat{z}_{i,t+1} - z_{i,t+1}\|\big] \nonumber\\
	&\textstyle \le N\vartheta_1 + {\vartheta}_2\sum_{i = 1}^N \sum_{t = 1}^T \mathbb{E}\bigg[ \frac{G_s \alpha_t}{\eta} + \frac{G_h \gamma_0}{\eta} \|[h_{i,t}(x_{i,t})]_+\| \bigg] \nonumber\\
	&\quad \textstyle+ N\sqrt{C}{\vartheta}_2\sum_{t = 1}^T s_t. \label{Lemma6-proof-iii-eq2}
\end{align}

Combining \eqref{Lemma6-proof-iii-eq1} and \eqref{Lemma6-proof-iii-eq2}, we deduce
\begin{align}
	&\textstyle \frac{1}{N}\sum_{i = 1}^N \sum_{t = 1}^T \mathbb{E}\big[\|[h_t(x_{i,t})]_+\|\big] \nonumber\\
	&\textstyle \le N G_h\vartheta_1 + \tilde{\psi}_8\sum_{t = 1}^T \alpha_t + \tilde{\psi}_9\sum_{i = 1}^N \sum_{t = 1}^T \mathbb{E}\big[\|[h_{i,t}(x_{i,t})]_+\|\big] \nonumber \\
	&\quad \textstyle + N\sqrt{C} G_h\vartheta_2\sum_{t = 1}^T s_t, \label{Lemma6-eq3}
\end{align}
where $\tilde{\psi}_8 = N G_h \vartheta_2 G_s / \eta$ and $\tilde{\psi}_9 = 1 + G_h^2 \vartheta_2 \gamma_0 / \eta$.

Invoking Slater's condition in Assumption~\ref{ass.slater},  we have
\begin{align}
	&\textstyle \mathbb{E}\big[ \sum_{i = 1}^N \sum_{t = 1}^T \frac{\varphi_{i,t+1}^T h_{i,t}(\bar{x})}{\gamma_t} \big] \nonumber\\
	&\le -\varsigma_s \sum_{i = 1}^N \sum_{t = 1}^T \mathbb{E}\big[\|[h_{i,t}(x_{i,t})]_+\|_1\big] \nonumber\\
	&\textstyle \le -\varsigma_s \sum_{i = 1}^N \sum_{t = 1}^T \mathbb{E}\big[\|[h_{i,t}(x_{i,t})]_+\|\big]. \label{Theorem1-proof-iii-eq1}
\end{align}

Selecting $\hat{x}_t = \bar{x}$ in \eqref{Lemma5-eq2} and utilizing \eqref{Theorem1-proof-iii-eq1}, it holds that
\begin{align}
\textstyle 	\varsigma_s \sum_{i = 1}^N \sum_{t = 1}^T \mathbb{E}\big[\|[h_{i,t}(x_{i,t})]_+\|\big] \le \tilde{\Lambda}_T(\{\bar{x}\}). \label{Theorem1-proof-iii-eq2}
\end{align}

 Applying the established parameter bounds \eqref{Theorem1-proof-ii-eq1}--\eqref{Theorem1-proof-ii-eq4} alongside \eqref{eq:alpha_sum} and \eqref{eq:s_alpha_sum} into $\tilde{\Lambda}_T(\{\bar{x}\})$ and dividing both sides by $\varsigma_s$, \eqref{Theorem1-proof-iii-eq2} yields
\begin{align}
	&\textstyle \sum_{i = 1}^N \sum_{t = 1}^T \mathbb{E}\big[\|[h_{i,t}(x_{i,t})]_+\|\big] \nonumber\\
	&\textstyle \le \frac{NK d(\mathcal{X})}{\varsigma_s\gamma_0} + \frac{2NG_s d(\mathcal{X}) \alpha_0}{\varsigma_s(1 - \theta_1)\gamma_0} T^{1 - \theta_1} + \frac{2N G_s^2 \alpha_0^2}{\varsigma_s\eta(1 - \theta_1)\gamma_0} T^{1 - \theta_1} \nonumber\\
	&\quad\textstyle  + \frac{N(\theta_1 + 1)}{\gamma_0 \varsigma_s\theta_1} d(\mathcal{X}) L_{\psi}\sqrt{C} s_0 + \frac{2N(\theta_1 + 1)}{\gamma_0 \varsigma_s(2\theta_1 + 1)} L_{\psi}  C s_0^2. \label{Theorem1-proof-iii-eq3}
\end{align}

Substituting (\ref{Theorem1-proof-iii-eq3}) into (\ref{Lemma6-eq3}), it follows that (\ref{thm1Cvv22}) holds.
\end{proof}

\subsection{Proof of Theorem~\ref{thm2}}\label{CCCC}
\begin{proof}
(i) From (\ref{KKey}) and (\ref{eq:v_bound_dist_res_norm}), we have
\begin{align}
	&\textstyle \varUpsilon_{i,t}  \le \left(L_{f,x} + \frac{C_{f,y}L_{g,xy}}{\mu_g}\right)\big\|y_{i,t+1} - y_{i,t}^*(x_{i,t})\big\| \nonumber\\
	& \quad\textstyle  + C_{g,xy}\big\|v_{i,t+1} - v_{i,t}^*(x_{i,t})\big\| \nonumber\\
	&\le \textstyle  \tilde{\kappa}_F
	 \big\|y_{i,t+1} - y_{i,t}^*(x_{i,t})\big\| + C_{g,xy}(r_v+ \frac{C_{f,y}}{\mu_g})\rho_H^{M/2}, \label{eq:dist_45_norm}
\end{align}
where $\tilde{\kappa}_F = L_{f,x} + (C_{f,y}L_{g,xy})/\mu_g + C_{g,xy}L_v$.

For the term $S_3$ defined in (\ref{KK1}), by virtue of the inner-loop stopping criterion $\mu_g\|y_{i,t+1} - y_{i,t}^*(x_{i,t})\| \le \|\nabla_y g_{i,t}(x_{i,t}, y_{i,t+1})\| \le \delta$, we have
\begin{align}
\textstyle 	S_3 \le \left( \frac{\tilde{\kappa}_F d(\mathcal{X})\delta}{\mu_g} + d(\mathcal{X})C_{g,xy}\left(r_v + \frac{C_{f,y}}{\mu_g}\right)\rho_H^{M/2} \right) T. \label{eq:S3_bound_final}
\end{align}

By replacing the recursive bounds \eqref{Fkappa}--\eqref{KK2} with the per-round bounds \eqref{eq:dist_45_norm} and \eqref{eq:S3_bound_final}, \eqref{eq:S346_bound} simplifies to
\begin{align}
	\textstyle S_4 + S_5 \le \tilde{\psi}_{10} + \frac{\tilde{\psi}_{11}}{N} \sum_{i=1}^{N}\sum_{t=1}^{T} \mathbb{E}\left[\|\varepsilon_{i,t}^z\|\right], \label{eq:S45_bound}
\end{align}
with $\tilde{\psi}_{10} = ( d(\mathcal{X})L + 2G_s )\vartheta_1$ and $\tilde{\psi}_{11} = ( d(\mathcal{X})L + 2G_s )\vartheta_2$.

By the convexity of $f_t$, as in the proof of Theorem~\ref{thm1}, we have
\begin{align*}
\textstyle \mathbb{E}[\mathrm{Net}\text{-}\mathrm{Reg}(T)] \le \frac{1}{N} \sum_{i=1}^N \sum_{t=1}^T \mathbb{E}[\langle \nabla f_t(x_{i,t}), x_{i,t} - x_t^* \rangle].
\end{align*}

Following the derivation of (\ref{eq:epsilon_bound}), with $\tilde{\psi}_2$ replaced by $\tilde{\psi}_{11}$, and combining \eqref{Lemma5-eq1}, \eqref{KK1}, \eqref{eq:S3_bound_final}, \eqref{eq:S45_bound}, and \eqref{quer} with $\hat{x}_t=x_t^*$ for $t\in [T]$ and $\hat{x}_{T+1}=x_T^*$, we have
\begin{align}\label{eq:Net_Reg_Final_Thm2}
	&\textstyle  \mathbb{E}[\mathrm{Net\text{-}Reg}(T)] \le \tilde{\psi}_{12} + \tilde{\psi}_{13} T + \frac{(2G_s^2 + \tilde{\psi}_{11}^2)\alpha_0}{\eta(1 - \theta_1)} T^{1 - \theta_1} \nonumber\\
	&\quad \textstyle + \left( \frac{Kd(\mathcal{X}) }{\alpha_0} + \frac{K}{\alpha_0} P_T \right) T^{\theta_1} + \frac{d(\mathcal{X})L_{\psi} \sqrt{C}s_0}{\alpha_0} (1 + \ln T),
\end{align}
where 
\begin{align}
	\textstyle  \tilde{\psi}_{12} & \textstyle = \tilde{\psi}_{10} + L_{\psi}  C \frac{s_0^2}{\alpha_0} \frac{\theta_1 + 2}{\theta_1 + 1} + \sqrt{C}\tilde{\psi}_{11} s_0 \frac{\theta_1 + 1}{\theta_1}, \nonumber\\
	\textstyle \tilde{\psi}_{13} &\textstyle  = {d(\mathcal{X})} \left(\frac{\tilde{\kappa}_F \delta}{\mu_g} + C_{g,xy}(r_v+\frac{C_{f,y}}{\mu_g}) \rho_H^{M/2} \right). \nonumber
\end{align}

By setting $M =\max \{1, \lceil -\ln T / \ln \rho_H \rceil \} $, we have $\rho_H^{M/2} \le T^{-1/2}$. Furthermore, choosing $\delta = \mathcal{O}(T^{-1/2})$ ensures that $\tilde{\psi}_{13}T = \mathcal{O}(\sqrt{T})$, it follows that (\ref{thm2ret}) holds.

(ii) and (iii) proofs follow similarly to those in Theorem~\ref{thm1}.

(iv) To bound the total inner iterations $I_{net}$, we need to explicitly extract $K_{i,t}$. If $K_{i,t}=0$, the desired upper bound on $K_{i,t}$ holds trivially. Hence, it suffices to consider $K_{i,t} \ge 1$.
Since the inner-loop does not terminate at step $K_{i,t}-1$, the gradient norm strictly exceeds $\delta$. Leveraging Assumption~\ref{ass:inner}~(ii), we have
\begin{align}\label{La11}
\textstyle 	\|y_{i,t+1}^{K_{i,t}-1} - y_{i,t}^*(x_{i,t})\| &\textstyle \ge \frac{1}{L_{g,y}} \|\nabla_y g_{i,t}(x_{i,t}, y_{i,t+1}^{K_{i,t}-1})\| \nonumber\\
	&\textstyle > \frac{\delta}{L_{g,y}}. 
\end{align}

Under the stated stepsize condition, each inner-gradient step in (\ref{eq:y_inner_loop}) contracts the distance to $y_{i,t}^*(x_{i,t})$ by a factor at most $\rho_y = 1 - \beta\rho \in (0, 1)$.
\begin{align}\label{La22}
\textstyle \|y_{i,t+1}^{K_{i,t}-1} - y_{i,t}^*(x_{i,t})\| 
	\le {\rho_y}^{K_{i,t}-1} \|y_{i,t+1}^0 - y_{i,t}^*(x_{i,t})\|. 
\end{align}

Combining (\ref{La11}) and (\ref{La22}) yields
\begin{align}\label{La33}
\textstyle 	{\rho_y}^{-(K_{i,t}-1)} \le \frac{L_{g,y}}{\delta} \|y_{i,t+1}^0 - y_{i,t}^*(x_{i,t})\|. 
\end{align}

Taking the natural logarithm on both sides of (\ref{La33}) and applying the inequality $\ln(x) \le x$ for $x > 0$, we obtain
\begin{align}
\textstyle 	K_{i,t} \le 1 + \frac{\ln(\delta^{-1})}{\ln(\rho_y^{-1})} + \frac{L_{g,y}+\|y_{i,t+1}^0 - y_{i,t}^*(x_{i,t})\|}{\ln(\rho_y^{-1})} . \label{eq:dist_50_l1}
\end{align}

For $t\ge 2$, by the triangle inequality, the initialization distance $\|y_{i,t+1}^0 - y_{i,t}^*(x_{i,t})\|$ can be decomposed as
\begin{align}
	&\textstyle\|y_{i,t+1}^0 - y_{i,t}^*(x_{i,t})\| \nonumber \\
	&\textstyle \le \|y_{i,t+1}^0 - y_{i,t-1}^*(x_{i,t-1})\| + \|y_{i,t-1}^*(x_{i,t-1}) - y_{i,t}^*(x_{i,t-1})\| \nonumber \\
	&\quad \textstyle + \|y_{i,t}^*(x_{i,t-1}) - y_{i,t}^*(x_{i,t})\|. \label{eq:dist_51_decomp_l1}
\end{align}
Since $y_{i,t+1}^0 = y_{i,t}$, the first term of (\ref{eq:dist_51_decomp_l1}) is bounded by $\delta/\mu_g$ according to the stopping criterion in (\ref{DDDDD39}). By the $L_y$-Lipschitz continuity of $y_{i,t}^*(\cdot)$ given in (\ref{eq:lip_y_star}), the third term of (\ref{eq:dist_51_decomp_l1})  is bounded by $L_y \|x_{i,t} - x_{i,t-1}\|$. Thus, for $t\ge 2$, we have
\begin{align}
	&\textstyle\|y_{i,t+1}^0 - y_{i,t}^*(x_{i,t})\| \le \frac{\delta}{\mu_g} + L_y \|x_{i,t} - x_{i,t-1}\| \nonumber \\
	&\qquad \textstyle+ \|y_{i,t-1}^*(x_{i,t-1}) - y_{i,t}^*(x_{i,t-1})\|. \label{eq:dist_51_final_l1}
\end{align}

Substituting \eqref{eq:dist_51_final_l1} into \eqref{eq:dist_50_l1} and summing over $i \in [N]$ and $t \in [T]$, we have
\begin{align}
	&\textstyle \mathcal{I}_{net} = \sum_{i=1}^N  K_{i,1}+\sum_{i=1}^N \sum_{t=2}^T K_{i,t} \nonumber \\
	&\textstyle \le \sum_{i=1}^N  K_{i,1} +\sum_{i=1}^N \sum_{t=2}^T \left( 1 + \frac{\ln(\delta^{-1})+L_{g,y}}{\ln(\rho_y^{-1})} + \frac{\delta}{\mu_g \ln(\rho_y^{-1})} \right) \nonumber \\
	&\quad \textstyle+ \frac{ L_y}{\ln(\rho_y^{-1})} \sum_{i=1}^N \sum_{t=2}^T \|x_{i,t} - x_{i,t-1}\| \nonumber \\
	&\quad \textstyle + \frac{1}{\ln(\rho_y^{-1})} \sum_{t=2}^T \sum_{i=1}^N \|y_{i,t-1}^*(x_{i,t-1}) - y_{i,t}^*(x_{i,t-1})\|\nonumber\\
	&\textstyle \le\sum_{i=1}^N  K_{i,1}  +\sum_{i=1}^N \sum_{t=2}^T \left( 1 + \frac{\ln(\delta^{-1})+L_{g,y}}{\ln(\rho_y^{-1})} + \frac{\delta}{\mu_g \ln(\rho_y^{-1})} \right) \nonumber\\
	&\quad \textstyle +\frac{N L_yd(\mathcal{X})T}{\ln(\rho_y^{-1})} + \frac{N }{\ln(\rho_y^{-1})} H_{1,T}. \label{eq:inet_pre_sup}
\end{align}

From \eqref{eq:dist_50_l1} and \eqref{eq:inet_pre_sup}, it follows that \eqref{totalNum} holds.
\end{proof}

\bibliographystyle{IEEEtran}
\bibliography{BiOnlineOpti}

\ifbool{submitA}
{
	\vspace{-40pt}
	\begin{IEEEbiography}[{\includegraphics[width=1in,height=1.25in,clip,keepaspectratio]{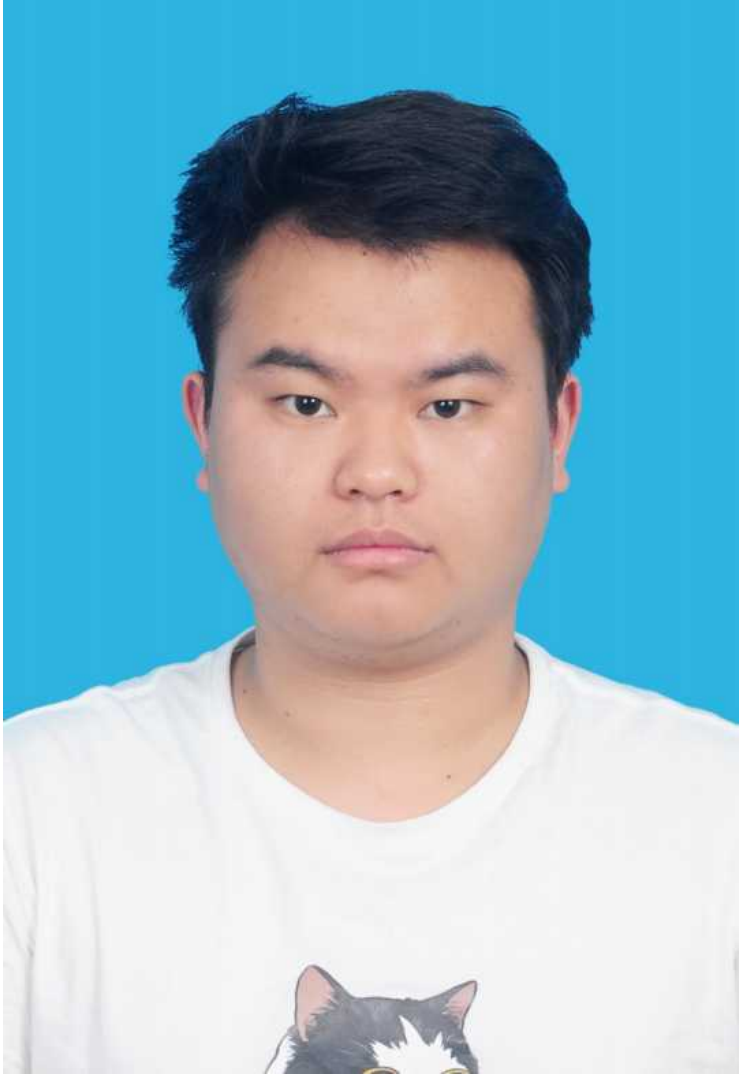}}]{Fan Li}
		received the B.S. degree in applied mathematics from Northwest Normal University, Lanzhou, China, in 2021, and the M.S. degree in control science and engineering from Central South University, Changsha, China, in 2024. He is currently pursuing the Ph.D. degree in control science and engineering at Northeastern University, Shenyang, China. His current research interests include distributed optimization and game theory, fixed point theory of operators.
	\end{IEEEbiography}
	\vspace{-40pt}
	\begin{IEEEbiography}[{\includegraphics[width=1in,height=1.25in,clip,keepaspectratio]{Xinlei_Yi.jpg}}]{Xinlei Yi} received the B.S. and M.S. degrees in mathematics from China University of Geoscience, Wuhan, China, and Fudan University, Shanghai, China, in 2011 and 2014, respectively, and the Ph.D. degree in electrical engineering from KTH Royal Institute of Technology, Stockholm, Sweden, in 2020. He was a Postdoc with KTH Royal Institute of Technology from 2020 to 2022 and with the Lab for Information and Decision Systems, Massachusetts Institute of Technology, Cambridge, MA, USA, from 2022 to 2024. He is a tenure-track professor of Shanghai Institute of Intelligent Science and Technology, Tongji University.
		
		Dr. Yi was selected as one of the four finalists for the 2021 European Systems and Control PhD Thesis Award. His current research interests include distributed online and optimization, online optimization, meta-learning and graph neural networks.
	\end{IEEEbiography}
	\vspace{-40pt}
	\begin{IEEEbiography}[{\includegraphics[width=1in,height=1.25in,clip,keepaspectratio]{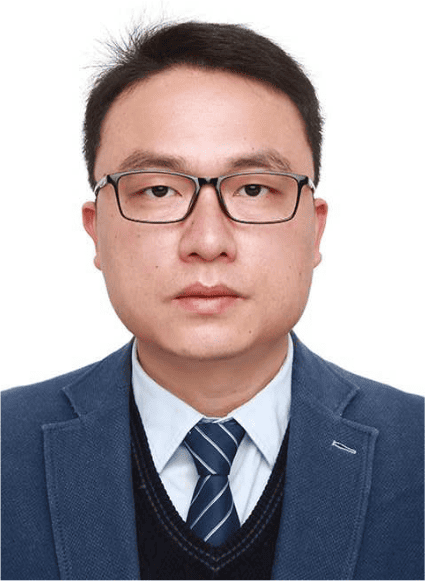}}]{Guanghui Wen} (Senior Member, IEEE) received the Ph.D. degree in mechanical systems and control from Peking University, Beijing, China, in 2012.
		
		Currently, he is an Endowed Chair Professor and the Vice Dean of the School of Automation, Southeast University, Nanjing, China. His current research interests include autonomous intelligent systems, complex networked systems, distributed control and optimization, resilient control, and distributed reinforcement learning.
		
		Prof. Wen was the recipient of the National Science Fund for Distinguished Young Scholars, the China Youth Science and Technology Award, the Australian Research Council Discovery Early Career Researcher Award, and the Asia Pacific Neural Network Society Young Researcher Award. He is an IET Fellow.
	\end{IEEEbiography}
	\vspace{-40pt}
	\begin{IEEEbiography}[{\includegraphics[width=1in,height=1.25in,clip,keepaspectratio]{Shiyang.jpg}}]{Yang Shi} (Fellow, IEEE) received the B.Sc. and first Ph.D. degrees in mechanical engineering and automatic control from Northwestern Polytechnical University, Xi’an, China, in 1994 and 1998, respectively, and the second Ph.D. degree in electrical and computer engineering from the University of Alberta, Edmonton, AB, Canada, in 2005.
		
		He was a Research Associate with the Department of Automation, Tsinghua University, Beijing, China,
		from 1998 to 2000. From 2005 to 2009, he was an Assistant Professor and an Associate Professor with the Department of Mechanical Engineering, University of Saskatchewan, Saskatoon, SK, Canada. In 2009, he joined the University of Victoria (UVic),
		Victoria, BC, Canada, where he is currently a Professor with the Department of Mechanical Engineering. His current research interests include networked and distributed systems, model predictive control, cyber-physical systems, robotics
		and mechatronics, navigation and control of autonomous systems, and energy system applications.
		
		Dr. Shi received the University of Saskatchewan Student Union Teaching Excellence Award in 2007, and the Faculty of Engineering Teaching Excellence Award in 2012 at UVic, on teaching and mentorship, and the 2023 REACH Award for Excellence in Graduate Student Supervision and Mentorship. On research, he is the recipient of the JSPS Invitation Fellowship
		(short-term) in 2013, the UVic Craigdarroch Silver Medal for Excellence in Research in 2015, the 2017 IEEE Transactions on Fuzzy Systems Outstanding Paper Award, the Humboldt Research Fellowship for Experienced Researchers in 2018, the CSME Mechatronics Medal in 2023, the IEEE Dr.-Ing. Eugene Mittelmann Achievement Award in 2023, the 2024 IEEE Canada Outstanding Engineer Award. He is IFAC Council Member; VP on Conference Activities of IEEE IES from 2022 to 2025 and the Chair of IEEE IES Technical Committee on IEEE INDUSTRIAL CYBERCPHYSICAL SYSTEMS. He is currently an Editor-in-Chief of IEEE TRANSACTIONS ON INDUSTRIAL ELECTRONICS, he also serves as Associate Editor for Automatica, IEEE TRANSACTIONS ON AUTOMATIC CONTROL, and Annual Review in Controls. He is a Fellow of ASME, CSME, Engineering Institute of Canada, Canadian Academy of Engineering, and a registered Professional Engineer in British Columbia, Canada.
	\end{IEEEbiography}
	\vspace{-40pt}
	\begin{IEEEbiography}[{\includegraphics[width=1in,height=1.25in,clip,keepaspectratio]{Tao_Yang.jpg}}]{Tao Yang} (Senior Member, IEEE) is a Professor at the State Key Laboratory of Synthetical Automation for Process Industries, Northeastern University. He was an Assistant Professor at the Department of Electrical Engineering, University of North Texas, Denton, USA, from 2016-2019. He received the Ph.D. degree in electrical engineering from Washington State University in 2012. Between August 2012 and August 2014, he was an ACCESS postdoctoral researcher with the ACCESS Linnaeus Centre, Royal Institute of Technology, Sweden. He then joined the Pacific Northwest National Laboratory as a postdoc, and was promoted to Scientist/Engineer II in 2015.
		
		His research interests include industrial artificial intelligence, integrated optimization and control, distributed control and optimization with applications to process industries, cyber physical systems, networked control systems, and multi-agent systems. He is an Associate Editor for IEEE Transactions on Control of Network Systems, IEEE Transactions on Control Systems Technology, and IEEE Transactions on Neural Networks and Learning System. He received Ralph E. Powe Junior Faculty Enhancement Award and Best Student Paper award (as an advisor) of several international conferences.
	\end{IEEEbiography}
}{}

\end{document}